\documentclass[12pt, reqno]{amsart}

\usepackage[utf8]{inputenc}
\usepackage[T2A]{fontenc} 
\usepackage[english]{babel}
\usepackage{textcase} 

\usepackage[margin=2cm]{geometry}

\usepackage{amssymb}
\usepackage{enumitem}
\usepackage{graphicx}
\graphicspath{{}{figs/}}
\usepackage[svgnames, dvipsnames]{xcolor}
\usepackage[all]{xy}

\usepackage{hyperref}
\hypersetup{
    hypertexnames=false,
    colorlinks,
    linkcolor={red!50!black},
    citecolor={blue!50!black},
    urlcolor={blue!80!black}
}

\newtheorem{corollary}[subsection]{Corollary}

\newtheorem{subtheorem}[subsubsection]{Theorem}
\newtheorem{sublemma}[subsubsection]{Lemma}

\newtheorem{subcorollary}[subsubsection]{Corollary}
\newtheorem{subremark}[subsubsection]{Remark}
\newtheorem{subexample}[subsubsection]{Example}
\newtheorem{subdefinition}[subsubsection]{Definition}

\newtheorem*{subproposition*}{Proposition}

\makeatletter
\@addtoreset{subsection}{section}
\@addtoreset{equation}{section}
\@addtoreset{figure}{section}
\@addtoreset{table}{section}
\makeatother

\newcommand\term[1]{\textit{\sffamily#1}}

\newcommand\Aman{A}
\newcommand\Bman{B}
\newcommand\Cman{C}

\newcommand\Eman{E}

\newcommand\Hman{H}

\newcommand\Jman{J}
\newcommand\Kman{K}

\newcommand\Mman{M}
\newcommand\Nman{N}

\newcommand\Pman{P}
\newcommand\Qman{Q}

\newcommand\Sman{S}
\newcommand\Tman{T}
\newcommand\Uman{U}
\newcommand\Vman{V}
\newcommand\Wman{W}
\newcommand\Xman{X}
\newcommand\Yman{Y}

\newcommand\bN{\mathbb{N}}

\newcommand\bR{\mathbb{R}}

\newcommand\restr[2]{#1|_{#2}}
\newcommand\Int[1]{\mathrm{Int}(#1)}
\newcommand\Cl[1]{\overline{#1}}

\newcommand\mycolor[1]{}  

\newcommand\todo[1]{{}}

\newcommand\hcl[2][\empty]{\mathrm{hcl}\ifx#1\empty\relax\else_{#1}\fi(#2)}

\newcommand\BranchSet[1]{#1_{br}}
\newcommand\HausdSet[1]{#1_{hs}}

\newcommand\ChNSQ[1][\empty]{\Gamma_{#1}} 
\newcommand\BrChNSQ[1][\empty]{\Gamma^{0}_{#1}}

\newcommand\Aone{{\mycolor{Mahogany}\Mman}}        
\newcommand\Abr{{\mycolor{Mahogany}\BranchSet{\Mman}}}    
\newcommand\Ahaus{{\mycolor{Mahogany}\HausdSet{\Mman}}}    
\newcommand\Aedg{{\mycolor{Mahogany}\Aone^{1}}} 
\newcommand\Avrt{{\mycolor{Mahogany}\Aone^{0}}} 

\newcommand\Acomp[1][\empty]{{\mycolor{Mahogany}\Jman_{#1}}} 
\newcommand\AcompImg[2]{\Acomp[(#1;#2)]}  
\newcommand\AcompCImg[2]{\Acomp[{[#1;#2]}]} 

\newcommand\Achart[1][\empty]{{\mycolor{Mahogany}\phi_{#1}}}
\newcommand\Bchart[1][\empty]{{\mycolor{Mahogany}\psi_{#1}}}
\newcommand\apr[1][\empty]{{\mycolor{BlueViolet}\pi_{#1}}}

\newcommand\Bone{{\mycolor{Mahogany}\Nman}}        

\newcommand\Xone{{\mycolor{ForestGreen}\ChNSQ}}    
\newcommand\Xvrt{{\mycolor{ForestGreen}\BrChNSQ}}   
\newcommand\Xedg{{\mycolor{ForestGreen}\Gamma^{1}}}             

\newcommand\cbr{\approx}

\newcommand\chcl[1]{{\mycolor{red}[#1]}}

\newcommand\eps{\varepsilon}

\newcommand\trm[1]{{\mycolor{red}#1}}
\newcommand\bRp[1]{\bR^{n}_{+}}

\newcommand\lms[2][\empty]{%
    {\mycolor{red}%
        \delta_{#2\ifx#1\empty\relax\else,#1\fi}
    }%
}

\newcommand\epr{p}  

\newcommand\EmptyClass{\Pman}
\newcommand\CEman{{\mycolor{Mahogany}\Cl{\Eman}}}
\newcommand\OEman{{\mycolor{NavyBlue}\Eman}}

\newcommand\relat{\sim}
\newcommand\Acw{{\mycolor{ForestGreen}\Kman}}

\newcommand\func{f}
\newcommand\Erelat{\Eman/\!\!\relat\,}
\newcommand\CErelat{\CEman/\!\!\relat\,}

\newcommand\lbnd{s} 
\newcommand\mbnd{t} 
\newcommand\lComp{\Pman} 
\newcommand\mComp{\Qman} 

\newcommand\liComp{\Pman_{i,\lambda}} 
\newcommand\mjComp{\Qman_{j,\mu}} 
\newcommand\libnd{\bar{s}} 
\newcommand\mjbnd{\bar{t}} 

\newcommand\lSegm{\Sman} 
\newcommand\mSegm{\Tman} %

\begin{document}
\author{Ihor Vlasenko, Sergiy Maksymenko}
\email{vlasenko@imath.kiev.ua, maks@imath.kiev.ua}
\address{Institute of Mathematics of the National Academy of Sciences of Ukraine, Kyiv, Ukraine}

\title{One-dimensional non-Hausdorff manifolds and cell complexes}


\begin{abstract}
Say that a connected non-Hausdorff one-dimensional manifold $M$ is \emph{graph-like}, whenever the set $M_{br}$ of its non-Hausdorff (called ``branch'') points is locally finite and every connected component of its complement has a countable base.
We prove that for every graph-like manifold $M$ there exists a one-dimensional CW complex $K$ with the set of vertices $K^{(0)}$, a vertex $v\in K^{(0)}$ and a (surjective) quotient map $\pi\colon M \to K\setminus\{v\}$ such that $\pi^{-1}(K^{(0)}\setminus \{v\}) = M_{br}\cup\partial M$.
Conversely, existence of such quotient map $\pi$ and the assumption that $M_{br}$ is nowhere dense imply that $M$ is graph-like.

Moreover, in this case $K\setminus\{v\}$ is the minimal Hausdorff factor of $M$, that is, for every continuous map $f\colon M \to N$ into a Hausdorff space $N$ there exists a unique continuous map $\hat{f}\colon K\setminus\{v\}\to N$ such that $f = \hat{f}\circ\pi$.



\end{abstract}

\maketitle

\section{Introduction}\label{sect:introduction}
Recall that a \term{manifold} is usually defined as a locally Euclidean space $\Aone$ that is also Hausdorff and has a countable base.
The latter two properties are equivalent to the metrizability of $\Aone$ and to the possibility of embedding it into Euclidean spaces.
Although relatively little work has been devoted to non-metrizable, and in particular non-Hausdorff, manifolds, they very often appear implicitly as quotient spaces of certain decompositions of spaces, in particular of foliations on manifolds, although in most cases they are rather ``pathological'' and do not carry substantial information.
On the other hand, under certain assumptions on the ``regularity'' of the behaviour of the leaves of foliations, their leaf spaces sometimes allow one to ``encode'' certain combinatorial properties of the decompositions themselves.

Let us briefly list known publications on non-metrizable manifolds, interest in which has grown significantly in recent years.
In~\cite{HaefligerReeb:EM:1957}, the existence of smooth structures on the \term{non-Hausdorff letter $\mathbb{Y}$} was proved.
The so-called \term{long line} $\mathbb{X}$ and the \term{Prüfer manifold} were studied in~\cite[Appendix~A]{Spivak:DG:1:1979}.
In~\cite{Nyikos:AM:1992}, it was shown that $\mathbb{X}$ admits non-diffeomorphic $C^{\omega}$ (real-analytic) structures, but it remains unclear whether $\mathbb{X}$ admits non-diffeomorphic $C^{\infty}$ structures.
In~\cite{GartsideGauldGreenwood:PrAMS:2008}, groups of homeomorphisms of non-Hausdorff manifolds were studied.
In particular, it was proved that for every group $G$ there exists a (possibly non-Hausdorff) manifold whose homeomorphism group is isomorphic to $G$.
A systematic study of non-metrizable manifolds was carried out in the monograph~\cite{Gauld:NonMetrManif:2014}.
In~\cite{OConnel:TA:2023, OConnel:TA:2024}, vector bundle structures over non-Hausdorff manifolds were investigated, and in~\cite{LysynskyiMaksymenko:Wisla:2025, LysynskyiMaksymenko:TA:2026}, classifications of all differentiable structures on the \term{non-Hausdorff line with two origins $\mathbb{L}$} and on $\mathbb{Y}$ were obtained.

General properties of non-metrizable manifolds, in particular their homogeneity and groups of homeomorphisms, were studied in~\cite{BaillifGabard:PrAMS:2008, BaillifGabardGauld:ProcAMS:2014, Baillif:arxiv:2025}.

Physical applications of non-Hausdorff manifolds were discussed in~\cite{Hajicek:CMP:1971,HellerPysiakSasin:JMP:2011, Luc:FPh:2020, LucPlacek:PhilSci:2020, NeimanConnell:PhRevD:2024}.
In~\cite{Haucourt:MSCS:2025}, devoted to parallel computational processes (threads), it was shown that every oriented graph can be ``split'' at each vertex so that the resulting space becomes a non-Hausdorff manifold.

\subsection{Main result of the paper}
Let $\Kman$ be a one-dimensional CW complex, see Subsection~\ref{sect:1CW}, and let $\Pman \subset \Kman$ be some (possibly empty) set of its vertices ($0$-cells).
Then the complement $\Xone=\Kman\setminus\Pman$ is called an \term{open one-dimensional CW complex} or a \term{topological graph}.
If $\Xone$ ``\term{has edges without endpoints}'', i.e.\ there exists an open $1$-cell $e$ in $\Kman$ such that $\overline{e} \cap \Pman = \varnothing$, then the topological graph $\Xone$ will be called \term{open}.

Now let $\Aone$ be a one-dimensional \term{manifold}, that is, a topological space locally homeomorphic to $[0;1)$.
It is well known and easily proved that if $\Aone$ is connected, Hausdorff, and has a countable base, then it is homeomorphic to one of the following manifolds: $[0;1]$, $[0;1)$, $(0;1)$, $S^1$.
If it is non-Hausdorff, then one can distinguish points at which it ``fails to be Hausdorff''.
Namely, we say that points $x,y\in\Aone$ are \term{inseparable (by open neighbourhoods)} if any neighbourhoods of these points intersect.
A point $x\in\Aone$ that is inseparable from some other point is called a \term{branch point} of $\Aone$, see Subsection~\ref{sect:branch_points}.
If it is separable by open neighbourhoods from all points $y\in\Aone\setminus\{x\}$, then $x$ is called a \term{Hausdorff} point of $\Aone$.

Clearly, the relation \term{being inseparable} for points of $\Aone$ is reflexive and symmetric but not transitive.
Then one can consider the transitive closure of this relation: points $x,y\in\Aone$ are called \term{chain inseparable} if there exists a sequence $x=x_0,x_1,\ldots,x_k=y \in \Aone$ such that $x_i$ is inseparable from $x_{i+1}$ for all $i=0,1,\ldots,k-1$.
This relation is already an equivalence relation on $\Aone$, and we denote it by $\cbr$, see Remark~\ref{rem:relation_hcl_NH} concerning its connection with~\cite{Baillif:arxiv:2025}.

Let $\Xone = \Aone/\!\cbr$ be the corresponding quotient space endowed with the quotient topology with respect to the quotient map $\apr\colon\Aone\to\Xone$.
Denote by $\Abr$ the set of all branch points of $\Aone$, and let $\Ahaus=\Aone\setminus\Abr$ be its complement --- the set of all Hausdorff points of $\Aone$.
We also set
\begin{align*}
    \Avrt &:= \Abr\cup\partial\Aone,                            & \Xvrt  &:= \apr(\Avrt), \\
    \Aedg &:= \Aone\setminus\Avrt=\Ahaus\setminus\partial\Aone, & \Xedg &:= \apr(\Aedg).
\end{align*}

Suppose that $\Xone$ has the structure of a topological graph.
We say that this structure is \term{compatible with the relation $\cbr$ of chain inseparability} on $\Aone$ if $\Xvrt$ is the set of vertices and, accordingly, the connected components of $\Xedg$ are the edges of $\Xone$.

The following theorem is the main result of the paper.
\begin{subtheorem}\label{th:1cw_struct}
Let $\Aone$ be a connected manifold that is not homeomorphic to the circle $S^1$.
Then the following conditions are equivalent:
\begin{enumerate}[label={\rm(G\arabic*)}]
\item\label{enum:th:1cw_struct:loc}
$\Abr$ is locally finite, and every connected component of $\Aedg$ has a countable base;
\item\label{enum:th:1cw_struct:1cw}
$\Abr$ is nowhere dense, and $\Xone$ has a structure of a topological graph compatible with the relation of chain inseparability on $\Aone$.
\end{enumerate}

Under these conditions, $\Xone$ is the ``minimal Hausdorff factor of $\Aone$'' in the sense that $\Xone$ is Hausdorff and for every continuous map $f\colon\Aone\to\Bone$ into a Hausdorff space $\Bone$ there exists a unique continuous map $\hat{f}\colon\Xone\to\Bone$ that makes the following diagram commutative:
\[
\xymatrix{
    \Aone \ar[rr]^{f} \ar[rd]_{\apr} & & \Bone \\
    & \Xone \ar[ur]_{\hat{f}}
}
\]
\end{subtheorem}

\begin{subdefinition}
A connected one-dimensional manifold that is not homeomorphic to the circle will be called \term{graph-like} if it satisfies one of the equivalent conditions~\ref{enum:th:1cw_struct:loc} or~\ref{enum:th:1cw_struct:1cw} of Theorem~\ref{th:1cw_struct}.
\end{subdefinition}

\begin{subremark}\rm
\begin{enumerate}[label={\rm\alph*)}, wide]
\item
If $\Abr$ is locally finite, then it is nowhere dense.
\item
Example~\ref{exmp:1manif:1_n_seq} shows that $\Abr$ can be nowhere dense and $\Xone$ can have a structure $\gamma$ of a topological graph, while $\Abr$ is not locally finite and $\gamma$ is not compatible with the relation $\cbr$.
\item
On the other hand, Examples~\ref{exmp:Abr_open} and~\ref{exmp:Xone_segment} below show that $\Abr$ can be dense (and hence not locally finite), while $\Xone$ has a structure of a topological graph compatible with $\cbr$.

\item
In the exceptional case $\Aone = S^1$, the circle is graph-like and Hausdorff, and the map $\apr\colon\Aone\to\Xone$ is a homeomorphism.
However, since the circle $\Xone$ ``has no vertices'', one cannot choose any canonical structure of a topological graph on it, although, of course, infinitely many such structures exist.
\end{enumerate}
\end{subremark}

\subsection{Examples of graph-like manifolds}\label{sect:examples}
\begin{subexample}\label{exmp:classif_1_manifs}
\rm
Recall that every connected Hausdorff one-dimensional manifold $\Aone$ with a countable base is homeomorphic to one of the following manifolds: $[0;1]$, $[0;1)$, $(0;1)$, $S^1$.
If it is different from the circle, we call it a \term{segment}.
Obviously, every segment is graph-like and has the structure of an open one-dimensional manifold whose vertices correspond to the boundary points.
\end{subexample}

\newcommand\Fol[1][\empty]{\mathcal{F}_{#1}}
\newcommand\RmSet{Q}
For $c\in\bR$, denote by $l_c = \bR\times\{c\}$ the horizontal line in $\bR^2$.
Then the family
\[
    \Fol = \{ l_{c} \mid c\in\bR \}
\]
of all these lines forms the ``standard'' one-dimensional foliation on $\bR^2$.
Let $\RmSet\subset\bR^2$ be an arbitrary closed subset.
Then $\bR^2\setminus\RmSet$ is an open subset of $\bR^{2}$, hence a $2$-manifold.
Denote by $\Fol[\RmSet]$ the partition of $\bR^2\setminus\RmSet$ into the connected components of the sets $l_{c}\setminus\RmSet$ for all $c\in\bR$.
This partition is also a one-dimensional foliation on the manifold $\bR^2\setminus\RmSet$.
It is easy to check that the corresponding quotient space $\Mman_{\RmSet} = (\bR^2\setminus\RmSet)/\Fol[\RmSet]$ is a one-dimensional manifold, but it is not necessarily Hausdorff.

\begin{figure}[htbp!]
\begin{tabular}{ccc}
\includegraphics[height=4cm]{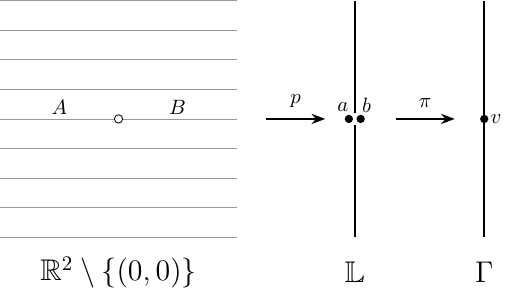} & \qquad\qquad &
\includegraphics[height=4cm]{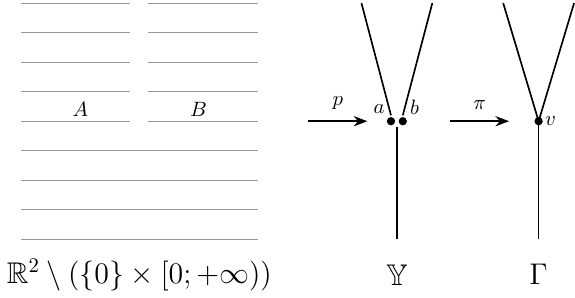} \\
(a) && (b)
\end{tabular}
\caption{The line with two origins $\mathbb{L}$ and the non-Hausdorff letter $\mathbb{Y}$}
\label{fig:line2origs}
\end{figure}

\begin{subexample}\label{exmp:1manif:L}
\rm
Let $\RmSet = \{(0,0)\}$ be the origin.
Then $\Mman_{\RmSet}$ is called the \term{line with two origins}, see Fig.~\ref{fig:line2origs}(a), and is denoted by $\mathbb{L}$.

Note that the point $(0,0)\in\bR^2$ splits the horizontal line $y=0$ into two open intervals $A$ and $B$, which correspond to two branch points $a$ and $b$ of the manifold $\mathbb{L}$.
Clearly, these points are inseparable from each other and are mapped to a single point $v$ in the quotient space $\Xone = \mathbb{L}/\cbr$.
Moreover, $\mathbb{L}$ is graph-like, and $\Xone$ is an open one-dimensional CW complex with one vertex $v$.

It is easy to see that $\mathbb{L}$ can also be obtained by gluing two copies of $\bR$ via the identity map of the set of nonzero numbers $\bR\setminus0$.
\end{subexample}

\begin{subexample}\label{exmp:1manif:Y}
\rm
Let $\RmSet = \{0\} \times [0;+\infty)$ be the nonnegative part of the $Oy$-axis, see Fig.~\ref{fig:letter_x}.
Then $\Mman_{\RmSet}$ is called the \term{non-Hausdorff letter $Y$} and is denoted by $\mathbb{Y}$, see Fig.~\ref{fig:line2origs}(b).
It also has exactly two branch points $a$ and $b$.
These points are inseparable from each other and give a single point $v$ in the quotient space $\Xone = \mathbb{Y}/\cbr$.
Again, $\mathbb{Y}$ is graph-like, and $\Xone$ is an open one-dimensional CW complex with one vertex $v$.

Also note that $\mathbb{Y}$ can be obtained by gluing two copies of $\bR$ via the identity map of the (open) set of \emph{negative} numbers.
\end{subexample}

\begin{subexample}\label{exmp:1manif:X}
\rm
Let
\[
    \RmSet =\{(0,1)\} \cup  \bigl(\{-1\} \times (\infty;0]\bigr) \cup \bigl(\{1\}\times[0;+\infty)\bigr)
\]
be the union of the point $(0,1)$ and two vertical rays starting on the $Ox$-axis, one pointing upward and the other downward.
The corresponding space $\Mman_{\RmSet}$ will be denoted by $\mathbb{X}$.

\begin{figure}[htbp!]
\includegraphics[height=5cm]{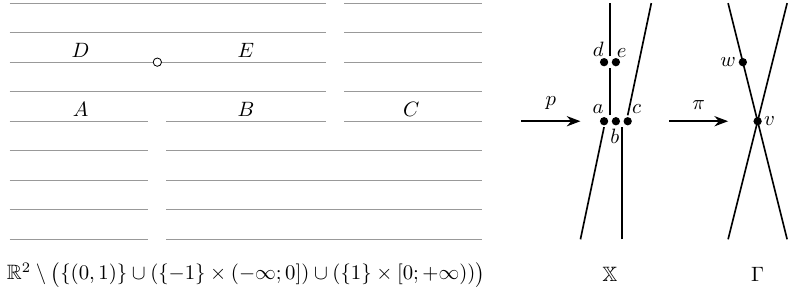}
\caption{The manifold $\mathbb{X}$}
\label{fig:letter_x}
\end{figure}

Note that the points $(0,-1)$ and $(0,1)$ split the $Ox$-axis into three open intervals $A$, $B$, $C$.
They correspond to three branch points $a$, $b$, $c$ in $\mathbb{X}$.
It is easy to see that the pairs $a$ and $b$, and $b$ and $c$ are inseparable.
On the other hand, $a$ and $c$ are separable by open neighbourhoods, but they are chain inseparable, and in the quotient space $\Xone = \mathbb{X}/\cbr$ all three points $a$, $b$, $c$ merge into one point $v$.

Similarly, the point $(0,1)$ splits the interval $(-\infty,1)\times\{1\}$ into two intervals $D$ and $E$.
These correspond to two inseparable branch points $d$ and $e$ in $\mathbb{X}$, which also merge into one point $w$ in $\Xone$.

Clearly, $\mathbb{X}$ is graph-like, and $\Xone$ is an open one-dimensional manifold homeomorphic to the letter $X$ and acquires the structure of a topological graph with vertices $v$ and $w$, which correspond to the chain inseparability classes of the branch points of $\mathbb{X}$.
\end{subexample}

\begin{subexample}[A non-graph-like manifold whose set of branch points is not locally finite]\label{exmp:1manif:1_n_seq}
\rm
Let $\Aman = \{(0,0)\} \cup \{(0,1/n) \mid n \in \bN\}$ be the origin together with a sequence on the $Oy$-axis converging to it.
In this case, each point $(0, c)\in \Aman$ divides the line $y=c$ into two open intervals, which give a pair of inseparable points in the manifold $\Mman_{\Aman}$.
It follows that the set $\Abr$ of branch points of $\Mman_{\Aman}$ consists of such pairs, which ``accumulate'' near the pair corresponding to the splitting of the $Ox$-axis by the origin.
In particular, $\Abr$ is not locally finite (i.e.\ the first requirement of condition~\ref{enum:th:1cw_struct:loc} is violated), and therefore $\Mman_{\Aman}$ is not graph-like.

At the same time, $\Abr$ is nowhere dense, i.e.\ the first requirement of condition~\ref{enum:th:1cw_struct:1cw} is satisfied.

On the other hand, the projection of the plane onto the $Oy$-axis induces a homeomorphism from the quotient space $\Xone = \Mman_{\Aman}/\cbr$ onto the real line $\bR$.
Thus $\Xone$ is a topological graph.
But when we identify $\Xone$ with $\bR$, the image $\Xvrt := \apr(\Abr) =\{1/n\}_{n\in\bN} \cup \{0\}$ is a convergent sequence together with its limit, and therefore there is no cell decomposition of $\bR$ for which $\Xvrt$ would be the set of vertices.
In other words, the second requirement of condition~\ref{enum:th:1cw_struct:1cw} is violated.
\end{subexample}

\subsection{A general construction of manifolds}
\label{sect:manif_general_construction}

\newcommand\LTwo[1]{|#1|^2}
\newcommand\Ua[2]{U_{#1,#2}}
\newcommand\phom[2]{\phi_{#1,#2}}
\newcommand\iaa{a}
\newcommand\ibb{b}
\newcommand\icc{c}
\newcommand\Msp[1]{\Eman_{#1}}

Let $(\Lambda,\leq)$ be an arbitrary partially ordered set.
For each $\iaa\in\Lambda$ fix some topological space, and let $\Eman = \mathop{\sqcup}\limits_{\iaa\in\Lambda} \Msp{\iaa}$ be the disjoint union of these spaces, i.e.\ a subset $\Uman\subset\Eman$ is open if and only if its intersection $\Uman\cap\Msp{\iaa}$ is open in $\Msp{\iaa}$ for all $\iaa\in\Lambda$.
For each pair $\iaa < \ibb \in \Lambda^2$ fix an \term{open embedding}
\[
    \Msp{\iaa}       \supset
    \Ua{\iaa}{\ibb}  \xrightarrow{~\phom{\iaa}{\ibb}~}
    \Msp{\ibb}
\]
of some open subset $\Ua{\iaa}{\ibb} \subset \Msp{\iaa}$.
In particular, $\Ua{\iaa}{\ibb}$ may be empty, in which case $\phom{\iaa}{\ibb}$ is in fact not defined.

We introduce the following equivalence relation $\sim$ on $\Eman$.
Let $x, y\in \Eman$.
We say that $x\sim y$ if and only if either $x=y$ or there exists a finite sequence
$x = x_1, x_2, \ldots, x_m = y \in \Eman$ such that for each $i=1,\ldots,m-1$
\begin{itemize}[itemsep=1ex]
\item
either $x_i \in\Ua{\iaa}{\ibb}$ and $x_{i+1} = \phom{\iaa}{\ibb}(x_i)$ for some $\iaa<\ibb\in\Lambda$ depending on $i$;
\item
or $x_{i} \in \phom{\iaa}{\ibb}(\Ua{\iaa}{\ibb})$ and $x_{i+1} = \phom{\iaa}{\ibb}^{-1}(x_{i})$  for some $\iaa<\ibb\in\Lambda$ depending on $i$.
\end{itemize}
In other words, one can pass from the point $x$ to the point $y$ in a finite number of steps by means of the embeddings $\phom{\iaa}{\ibb}$ or their inverses.

Let further $\Aone = \Eman/\!\sim$ be the set of equivalence classes and $p\colon\Eman\to\Aone$ the corresponding quotient map.
We endow $\Aone$ with the quotient topology with respect to $p$.
We say that $\Aone$ is \term{glued from the family $\{\Msp{\iaa}\}_{\iaa\in\Lambda}$ by means of the embeddings $\{\phom{\iaa}{\ibb}\}_{\iaa<\ibb\in\Lambda}$}.

Note that we do not require ``functoriality'', such as $\phom{\iaa}{\icc} = \phom{\ibb}{\icc} \circ \phom{\iaa}{\ibb}$, and therefore the restriction $\restr{p}{\Msp{\iaa}}\colon\Msp{\iaa}\to\Aone$ is not always injective.
\begin{sublemma}\label{lm:gluing_open_map}
The quotient map $p\colon\Eman\to\Aone$ is open.
\end{sublemma}
\begin{proof}
Let $\Vman\subset\Eman$ be an arbitrary open set.
Set $\Vman_0 = \Vman$ and further, by induction, for $i\geq0$ put
\[
    \Vman_{i+1} =
        \Vman_{i} \ \bigcup
        \mathop{\cup}\limits_{\iaa < \ibb \in \Lambda}
                \phom{\iaa}{\ibb}(\Vman_{i} \cap \Ua{\iaa}{\ibb})
        \ \ \bigcup \
        \mathop{\cup}\limits_{\iaa < \ibb \in \Lambda}
                \phom{\iaa}{\ibb}^{-1}
                    (\Vman_{i} \cap \phom{\iaa}{\ibb}(\Ua{\iaa}{\ibb}))
\]
Then it follows from the definition of the relation $\sim$ that
\[
    p^{-1}(p(\Vman)) = \bigcup_{i\geq0} \Vman_{i}.
\]
Since $\Vman$ and $\Ua{\iaa}{\ibb}$ are open, and $\phom{\iaa}{\ibb}$ are open embeddings for all $\iaa<\ibb\in\Lambda$, the set $\Vman_i$ is open for all $i\geq0$.
Therefore $p^{-1}(p(\Vman))$ is also open in $\Eman$.
Now, since $p$ is a quotient map, it follows that $p(\Vman)$ is open in $\Aone$.
\end{proof}

\begin{subcorollary}\label{cor:glue_a_manifold}
Suppose that for each $\iaa\in\Lambda$ the restriction $\restr{p}{\Msp{\iaa}}\colon\Msp{\iaa}\to\Aone$ is injective, i.e.\ $\Msp{\iaa}$ consists of pairwise non-equivalent points.
Then this restriction is an open embedding.

In particular, if for some $n\geq1$ the space $\Msp{\iaa}$ is an $n$-dimensional manifold, then $\Aone$ is also an $n$-dimensional manifold.
\end{subcorollary}
\begin{proof}
Since $\Msp{\iaa}$ is open in $\Eman$, the restriction $\restr{p}{\Msp{\iaa}}$ is an open injective map, and hence a homeomorphism onto its image $\Wman_{\iaa}:=p(\Msp{\iaa})$, which is an open set in $\Aone$.

If now each $\Msp{\iaa}$ is an $n$-manifold, then $\Aone$ is the union of the open subsets $p(\Msp{\iaa})$, $\iaa\in\Lambda$, each of which is homeomorphic to an $n$-manifold, and therefore $\Aone$ is also an $n$-manifold.
\end{proof}

\subsection{``Shortcuts'' between points in one-dimensional manifolds}
The following Examples~\ref{exmp:one_portal} and~\ref{exmp:many_portals} are taken from \cite[Example~7.1]{Baillif:arxiv:2025}.
They generalize Example~\ref{exmp:1manif:X} and are needed in order to illustrate the independence of the requirements in condition~\ref{enum:th:1cw_struct:1cw}.

\begin{subexample}
\label{exmp:one_portal}
\rm
Let $\Msp{1}$ be a one-dimensional manifold and $x,y\in\Msp{1}$ two distinct points separable by open neighbourhoods.
We show how to add to $\Msp{1}$ a single point $z$ so that the resulting topological space $\Aone$ remains a one-dimensional manifold, but at the same time $z \in \hcl{x} \cap \hcl{y}$, i.e.\ in $\Aone$ the points $x$ and $y$ become chain inseparable.
Moreover, it turns out that this construction can be applied simultaneously even to an \term{uncountable} number of pairs of points of $\Aone$.

By assumption, there exist open embeddings $\alpha,\beta\colon (-1;1)\to\Msp{1}$ such that $x=\alpha(0)$, $y=\beta(0)$, and their images $\alpha\bigl((-1;1)\bigr)$ and $\beta\bigl((-1;1)\bigr)$ are disjoint.

Let $\Msp{0} = (-1;1)$, $\Ua{0}{1} = \Msp{0}\setminus\{0\}=(-1;0)\sqcup(0;1)$, and $\phom{0}{1}\colon\Ua{0}{1}\to\Msp{1}$ the open embedding defined by the formula
\[
    \phom{0}{1}(t) =
    \begin{cases}
        \alpha(t), & t\in(-1;0) \\
        \beta(t),  & t\in(0;1).
    \end{cases}
\]
and let $\Aone = \Msp{0} \cup_{\phom{0}{1}} \Msp{1}$ be the space obtained by gluing $\Msp{0}$ and $\Msp{1}$ by means of $\phom{0}{1}$.
Let further $\Eman = \Msp{0} \sqcup \Msp{1}$ and $p\colon\Eman\to\Aone$ the quotient map.
Note that each restriction $\restr{p}{\Msp{i}}\colon\Msp{i}\to\Aone$, $i=0,1$, is injective%
\footnote{In fact, $\restr{p}{\Msp{1}}$ is always injective, and for the injectivity of $\restr{p}{\Msp{0}}$ it suffices to require that $\alpha\bigl((-1;0)\bigr) \cap \beta\bigl((0;1)\bigr) = \varnothing$.}, and therefore, by Corollary~\ref{cor:glue_a_manifold}, $\Aone$ is a one-dimensional manifold.

Identify $\Msp{1}$ with its image in $\Aone$ and let $z = \phom{0}{1}(0)$.
Then $\Aone = \Msp{1} \sqcup \{z\}$.
Moreover, $z\in\hcl{x}\cap\hcl{y}$, and therefore the points $x,y$ are chain inseparable.
The point $z$ will be called a \term{shortcut between $x$ and $y$}.

Note that in Example~\ref{exmp:1manif:X} the point $b$ is such a shortcut between $a$ and $c$.
\end{subexample}

\newcommand\CLambda{\Omega}
\begin{subexample}[{\rm\cite[Example~7.1]{Baillif:arxiv:2025}}]
\label{exmp:many_portals}
\rm
\newcommand\asymbol{*}
More generally, let $\asymbol$ be some symbol not belonging to $\bR$,
\[
    \CLambda \subset \{ (x,y)\in\bR^2 \mid x,y\in\bR, x<y\}
\]
an arbitrary subset of the plane consisting of points $(x,y)$ with $x<y$, and $\Lambda = \CLambda \cup \{\asymbol\}$.

We introduce the following partial order on $\Lambda$:
for each pair $(x,y)\in\CLambda$ let $(x,y) < \asymbol$, and moreover $(x,y)$ is not comparable with any other pair $(x',y')\in\CLambda$.

Set $\Msp{\asymbol} = \bR$.
Now, for each pair $(x,y)\in\CLambda$, let
\begin{align*}
    \Msp{(x,y)} &:= (-1;1), &
    \Ua{(x,y)}{\asymbol} &:= (-1;1) \setminus \{0\} = (-1;0) \cup (0;1),
\end{align*}
and let $\phom{(x,y)}{\asymbol}\colon\Ua{(x,y)}{\asymbol}\to\bR$ be the open embedding defined by the formula:
\[
    \phom{(x,y)}{\asymbol}(t) =
    \begin{cases}
        x+t, & t\in(-1;0) \\
        y+t, & t\in(0;1).
    \end{cases}
\]
As in Subsection~\ref{sect:manif_general_construction}, let $\Eman = \Msp{\asymbol} \sqcup \mathop{\sqcup}\limits_{(x,y)\in\CLambda}\Msp{(x,y)}$ be the disjoint union of all these spaces, $\sim$ the corresponding equivalence relation on $\Eman$, $\Aone_{\CLambda} = \Eman/\!\sim$ the quotient space obtained by gluing $\Eman$ by means of the family $\{\phom{(x,y)}{\asymbol}\}_{(x,y)\in\CLambda}$, and $p\colon\Eman\to\Aone_{\CLambda}$ the quotient map.

Note that for each pair $(x,y)\in\CLambda$ the embedding $\phom{(x,y)}{\asymbol}$ is \term{monotonically increasing}.
It follows that the points of $\Msp{(x,y)}$ are pairwise non-equivalent, i.e.\ $\restr{p}{\Msp{(x,y)}}$ is injective, and therefore, by Corollary~\ref{cor:glue_a_manifold}, $\Aone_{\CLambda}$ is a one-dimensional manifold.

By construction, $\phom{(x,y)}{\asymbol}$ identifies $\Ua{(x,y)}{\asymbol}=(-1;1)\setminus\{0\}$ with an open subset of $\bR = \Msp{\asymbol}$, and the point $z_{x,y} = \restr{p}{\Msp{(x,y)}}(0)$ is a ``shortcut'' between $x$ and $y$ in $\Aone_{\CLambda}$.
In particular, each pair $(x,y)\in\CLambda$ adds to $\bR$ in $\Aone_{\CLambda}$ a single point $z_{x,y}$, and we have a canonical bijection  $\bR\sqcup\CLambda \equiv \Aone_{\CLambda}$:
\begin{align}\label{equ:R_Omega__M_Omega}
    &\bR   \ni t     \longmapsto p(t)    \in \Aone_{\CLambda}, &
    &\CLambda\ni (x,y) \longmapsto z_{x,y} \in \Aone_{\CLambda}.
\end{align}
In other words, we obtain the structure of a one-dimensional manifold on $\bR\sqcup\CLambda$ by adding shortcuts between the points $x,y$ for all $(x,y)\in\CLambda$.
In particular, each such pair of points $x$ and $y$ is chain inseparable.
\end{subexample}

As a corollary, we obtain the following lemma.
\begin{sublemma}\label{lm:one_chain_nonsep_class}
Let $\Qman\subset\bR$ be an arbitrary subset containing more than one point.
Then there exist a one-dimensional manifold $\Aone$ and an open embedding $\phi\colon\bR\to\Aone$ such that $\phi(Q) \cup \bigl(\Aone\setminus\phi(\bR)\bigr)$ forms a single chain inseparability class of $\Aone$.
\end{sublemma}
\begin{proof}
Let $x\in\Qman$ be an arbitrary point, and let $\Aone = \bR\sqcup(\Qman\setminus\{x\})$ be the one-dimensional manifold obtained by adding to $\bR$ a ``shortcut'' $z_{x,y}$ between $x$ and each of the points $y\in\Qman\setminus\{x\}$.
Then the chain inseparability class $\chcl{x}$ of the point $x$ is $\Qman \cup (\Aone\setminus\bR)$.
On the other hand, all other points $w \in \bR\setminus\Qman \subset\Aone$ are Hausdorff.
\end{proof}
\begin{subexample}\label{exmp:Abr_open}\rm
Let $\Qman = \bR$ and $\Aone := \Aone_{\CLambda}$ the manifold from Lemma~\ref{lm:one_chain_nonsep_class}.
Then all its points are pairwise chain inseparable.
Therefore
\begin{itemize}[leftmargin=*]
\item $\Abr = \Aone$, and hence this set \term{is not nowhere dense},
\item on the other hand, the quotient space $\Xone = \Aone/\cbr$ is a point, and hence has the structure of a topological graph.
\end{itemize}
Thus the first requirement in condition~\ref{enum:th:1cw_struct:1cw} is violated, while the second one is satisfied.
\end{subexample}

\begin{subexample}\label{exmp:Xone_segment}\rm
Let $\Nman = (-1;2)$ and let $\Aone$ be the one-dimensional manifold obtained by adding to $\Nman$ shortcuts
\begin{itemize}
\item
between the point $0$ and each point $t\in(-1;0)$;
\item
and also between the point $1$ and each point $t\in(1;2)$.
\end{itemize}
Then $\Abr$ has exactly two chain inseparability classes, of the points $0$ and $1$, while all other points $t\in(0;1)$ are Hausdorff.
In particular, $\Xone$ can be identified with $[0;1]$ so that $0 \in [0;1]$ corresponds to the chain inseparability class of the point $0\in\Aone$, and $1 \in [0;1]$ to the chain inseparability class of the point $1\in\Aone$.
Thus the cell decomposition $\Xone = [0;1]$, which has exactly two vertices $0$ and $1$, is compatible with the inseparability relation $\cbr$.

Hence, as in Example~\ref{exmp:Abr_open}, every non-trivial chain inseparability class has non-empty interior, i.e.\ the first requirement of condition~\ref{enum:th:1cw_struct:1cw} is violated, but the quotient space $\Xone$ has a structure of a topological graph compatible with $\cbr$.
\end{subexample}

\subsection{Structure of the paper}
Section~\ref{sect:preliminaries} contains definitions and elementary properties of quotient maps, locally finite families of sets, branch points and chain inseparability of spaces, one-dimensional manifolds, and one-dimensional CW complexes.
In Section~\ref{sect:open_embeddings_01} we prove several general statements about embeddings of open intervals into topological spaces and the possibility of extending such embeddings to the ends of these intervals.
Section~\ref{sect:proof:th:1cw_struct} contains the proof of Theorem~\ref{th:1cw_struct}.

\section{Preliminaries}\label{sect:preliminaries}

\subsection{Quotient maps}
Recall that a map $\func\colon\Xman\to\Yman$ between topological spaces is called a \term{quotient map}%
\footnote{In this definition we do not require $\func$ to be surjective---in most of the statements of this paper, in particular in Lemma~\ref{lm:qmaps}, surjectivity of a quotient map is not needed. In the cases where it is needed, we require it explicitly.}, if it has the following property: a subset $\Bman\subset\Yman$ is open if and only if its preimage $\func^{-1}(\Bman)$ is open.
It follows from the definition that every quotient map is continuous.
The following well-known and elementary lemma contains some properties of quotient maps that will be used in this paper.

\begin{sublemma}\label{lm:qmaps}
Let $\func\colon\Xman\to\Yman$ be a quotient map between topological spaces.
Then the following statements hold.
\begin{enumerate}[label={\rm(\alph*)}, leftmargin=*]
\item\label{enum:lm:qmaps:cont_compos}
Let $\xi\colon\Yman\to\Qman$ be a map into an arbitrary topological space.
Then $\xi$ is continuous if and only if the composition $\xi\circ\func\colon\Xman\to\Qman$ is continuous.
\item\label{enum:lm:qmaps:restr_to_inv_img}
Let $\Bman\subset\Yman$ be an arbitrary subset whose preimage $\func^{-1}(\Bman)$ is non-empty and open or closed in $\Xman$.
Then the restriction $\restr{\func}{\func^{-1}(\Bman)}\colon \func^{-1}(\Bman) \to \Bman$ is also a quotient map.
In particular, if $\restr{\func}{\func^{-1}(\Bman)}$ is a bijection, then it is a homeomorphism.
\end{enumerate}
\end{sublemma}

\subsection{Locally finite families of sets}
Let $\Aone$ be a topological space, $\Aman\subset\Aone$ a subset, and $x\in\Aman$ a point.
An open neighbourhood $\Uman$ of $x$ is called \term{isolating with respect to $\Aman$} if $\Uman\cap\Aman=\{x\}$.
Clearly, $x$ is isolated in $\Aman$ if and only if it has an isolating neighbourhood with respect to $\Aone$.

Recall also that a subset $\Aman\subset\Aone$ is called
\begin{itemize}[leftmargin=*]
\item \term{locally finite} if the family of all singletons from $\Aman$ is locally finite, i.e.\ each point $x\in\Aone$ has a neighbourhood whose intersection with $\Aman$ is finite;
\item \term{discrete} if it is discrete in the induced topology, i.e.\ each point $a\in \Aman$ has a neighbourhood $\Uman_{a}$ containing no other points of $\Aman$;
\item \term{strongly discrete} if for each point $a\in \Aman$ one can choose a neighbourhood $\Uman_{a}$ such that $\Uman_{a} \cap\Uman_{b} = \varnothing$ for $a\ne b \in \Aman$.
\end{itemize}

Recall also that a family of subsets $\beta=\{\Bman_{i}\}_{i\in I}$ of a topological space $\Aone$ is called
\begin{itemize}[leftmargin=*]
\item \term{locally finite} if each point $x\in\Aone$ has a neighbourhood that intersects only finitely many elements of $\beta$;
\item \term{discrete} if each $\Bman_i$ is clopen in the induced topology of $\Bman$, i.e.\ each $\Bman_i$ has an open neighbourhood that does not intersect the other elements of $\beta$;
\item \term{strongly discrete} if for each $i\in I$ there exists an open neighbourhood $\Uman_i$ of $\Bman_i$ such that $\Uman_i\cap\Uman_j=\varnothing$ for distinct $i,j\in I$.
\end{itemize}
Let $\Bman\in\beta$.
An open neighbourhood $\Uman$ of $\Bman$ is called \term{isolating with respect to $\beta$} if $\Uman$ does not intersect the other elements of $\beta$.
Clearly, the family $\beta$ is discrete if and only if each of its elements has an isolating neighbourhood with respect to $\beta$.

Let us note the following elementary statements.
\begin{sublemma}[{\rm e.g.~\cite[Lemma~2.3.1]{MaksymenkoPolulyakh:PIGC:2021}}]
\label{lm:T1_loc_fin__closed_discr}
Let $\Aone$ be a $T_1$ space.
Then a subset $\Aman\subset\Aone$ is locally finite if and only if $\Aman$ is closed and discrete.
\end{sublemma}

\begin{sublemma}\label{lm:open_disj_family}
Let $\{\Acomp[\lambda]\}_{\lambda\in\Lambda}$ be a family of open subsets of a topological space $\Aone$ such that $\Acomp[\lambda]\cap\Acomp[\mu]=\varnothing$ for $\lambda\ne\mu\in\Lambda$.
Denote $\Uman = \mathop{\cup}\limits_{\lambda\in\Lambda}\Acomp[\lambda]$.
Then $\Cl{\Acomp[\lambda]}\setminus\Acomp[\lambda] \subset \Aone\setminus\Uman$ for all $\lambda\in\Lambda$.
\end{sublemma}
\begin{proof}
Suppose that $\Cl{\Acomp[\lambda]}\setminus\Acomp[\lambda]$ intersects $\Uman$ and let $x\in(\Cl{\Acomp[\lambda]}\setminus\Acomp[\lambda]) \cap \Uman$ be any point.
Then $x \in \Acomp[\mu]$ for some $\mu\in\Lambda$.
But then the open set $\Acomp[\mu]$ intersects the closure $\Cl{\Acomp[\lambda]}$, hence it must also intersect the set $\Acomp[\lambda]$ itself, which contradicts the assumption.
\end{proof}

\begin{sublemma}\label{lm:strongly_discr_family}
Let $\Aone$ be a topological space and $\{\Bman_{i}\}_{i\in I}$ a discrete family of sets whose union $\Bman = \mathop{\cup}\limits_{i\in I} \Bman_i$ is closed (open).
Then each set $\Bman_i$ is also closed (open).
\end{sublemma}
\begin{proof}
By hypothesis, $\Bman_i$, $i\in I$, is clopen in $\Bman$.
Therefore, if $\Bman$ is closed (open) in $\Aone$, then its clopen subset $\Bman_i$ will also be closed (open) in $\Aone$.
\end{proof}

\begin{sublemma}\label{lm:discr_family_of_discr_sets}
Let $\Aone$ be a topological space and $\{\Bman_{i}\}_{i\in I}$ a discrete family of discrete sets.
Then their union $\Bman = \mathop{\cup}\limits_{i\in I} \Bman_i$ is a discrete set.
\qed
\end{sublemma}

\begin{sublemma}\label{lm:locfin_dicr_familiies_inverse}
Let $\epr\colon\Aone\to\Xone$ be a continuous map and $\gamma$ a locally finite (discrete, strongly discrete) family of subsets of $\Xone$.
Then its preimage $\beta = \{ \epr^{-1}(\Cman) \mid \Cman\in\gamma\}$ is a locally finite (discrete, strongly discrete) family of subsets of $\Aone$.
\qed
\end{sublemma}

\subsection{Branch points and chain inseparability classes}\label{sect:branch_points}
Let $(\Aone,\tau)$ be a topological space.
For each point $x\in\Aone$ we define its \term{Hausdorff closure}:
\[
    \hcl{x} := \mathop{\cap}\limits_{x\in\Vman\in\tau}\Cl{\Vman}
\]
as the intersection of the closures of all its neighbourhoods.
By definition, $x\in\hcl{x}$, but, in general, $\hcl{x}$ may contain points different from $x$.

More generally, for a subset $\Aman\subset\Aone$ and a point $x\in\Aman$, denote by $\hcl[\Aman]{x}$ its Hausdorff closure in $\Aman$ with the induced topology.

The following simple lemma holds, and we leave its proof to the reader.
\begin{sublemma}\label{lm:hcl_props}
Let $x,y\in\Aone$ be two points.
Then the following conditions are equivalent:
\begin{enumerate}[label={\rm(\alph*)}]
\item\label{enum:lm:hcl_props:y_in_hclx} $y \in \hcl{x}$;
\item\label{enum:lm:hcl_props:non-sep} $x$ and $y$ are \term{inseparable}, i.e.\ any of their neighbourhoods intersect;
\item\label{enum:lm:hcl_props:x_in_hcly} $x \in \hcl{y}$.
\end{enumerate}
Moreover, for an arbitrary subset $\Aman\subset\Aone$ and a point $x\in\Aman$,
\begin{equation}\label{equ:hclA_x}
    \hcl[\Aman]{x} \subset \Aman \cap \hcl{x}.
\end{equation}
\end{sublemma}

\begin{subdefinition}
A point $x\in\Aone$ is called a \term{Hausdorff} point of $\Aone$ if $\hcl{x} = \{x\}$, and a \term{branch point} of $\Aone$ if $\hcl{x} \setminus \{x\} \ne \varnothing$, i.e.\ in $\Aone\setminus\{x\}$ there exist points inseparable from $x$.
Denote by
\begin{align}
    \Ahaus &= \{ x\in\Aone \mid \hcl{x}=\{x\}\}, &
    \Abr   &= \Aone\setminus\Ahaus
\end{align}
the sets of all \trm{Hausdorff} points and all branch points of $\Aone$, respectively.
\end{subdefinition}

Thus, by definition, $\Aone$ is Hausdorff if and only if all its points are \trm{Hausdorff}.

Note that, by definition, $\hcl{x}$ is always a closed set, being an intersection of closed sets.

We shall need Lemma~2.10 and Example~7.1 from the preprint~\cite{Baillif:arxiv:2025}, which are related to the main theorem.
\begin{sublemma}[{\rm\cite[Lemma~2.10]{Baillif:arxiv:2025}}]
If $\Aone$ is locally Hausdorff, for example a manifold, then for each point $x\in\Aone$ the set $\hcl{x}\setminus\{x\}$ is closed and nowhere dense in $\Aone$.
In particular, $x$ is isolated in $\hcl{x}$.
\end{sublemma}
\begin{proof}
For the reader's convenience, we recall the arguments from~\cite[Lemma~2.10]{Baillif:arxiv:2025}.

Let $\Uman$ be an open Hausdorff neighbourhood of the point $x$.
We show that $\Uman\cap\hcl{x} = \{x\}$, i.e.\ $x$ is isolated in $\hcl{x}$.

Since $\hcl{x}$ is closed, this will prove that $\hcl{x}\setminus\{x\} = \hcl{x}\setminus\Uman$ is also closed.
Moreover, since $\hcl{x} \subset\Cl{\Uman}$, we obtain that $\hcl{x}\setminus\{x\} \subset \Cl{\Uman}\setminus\Uman$.
The latter set has empty interior, and therefore $\hcl{x}\setminus\{x\}$ also has empty interior.

Indeed, for each point $y\in\Uman\setminus\{x\}$ there exist neighbourhoods $\Vman_x, \Vman_y \subset \Uman$ of the points $x$ and $y$, respectively, such that $\Vman_x\cap\Vman_y=\varnothing$.
But then $\Uman_x:=\Uman\cap\Vman_x$ and $\Uman_y:=\Uman\cap\Vman_y$ are open disjoint neighbourhoods of these points in $\Aone$.
Therefore $y\not\in\hcl{x}$.
\end{proof}

From Lemma~\ref{lm:hcl_props}\ref{enum:lm:hcl_props:non-sep} it follows that the relation $y \in \hcl{x}$ (i.e.\ being inseparable) is reflexive and symmetric but, in general, not transitive.
One can then extend it to an equivalence relation on $\Aone$, which is called the \term{transitive closure} and is defined as follows.

We say that points $x,y\in\Aone$ are \term{\trm{chain inseparable}} if there exists a finite sequence of points $x_0=x, x_1,\ldots,x_k = y \in \Aone$ such that $x_i\in\hcl{x_{i+1}}$ for all $i=0,\ldots,k-1$.
In this case we write $x \cbr y$.
Obviously, the relation $\cbr$ of being \trm{chain inseparable}\ is an equivalence relation.
Denote by $\ChNSQ[\Aone]$ the corresponding partition of $\Aone$ into equivalence classes with respect to $\cbr$.
We introduce on $\ChNSQ[\Aone]$ the quotient topology with respect to the corresponding quotient map $\apr[\Aone]\colon\Aone\to\ChNSQ[\Aone]$.
Then the map $\apr[\Aone]$ becomes a quotient map.

For a subset $\Qman\subset\Aone$, the set $\chcl{\Qman} := \apr[\Aone]^{-1}(\apr[\Aone](\Qman))$ will be called the \emph{saturation} of $\Qman$ with respect to the relation~$\cbr$.
Also, $\Qman$ is called \emph{$\cbr$-saturated} if $\chcl{\Qman}=\Qman$.
In particular, for a point $x\in\Aone$, the set $\chcl{x}:= \apr[\Aone]^{-1}(\apr[\Aone](x))$ will be called the \term{chain inseparability class} of this point.

\begin{subremark}\label{rem:relation_hcl_NH}\rm
In~\cite{Baillif:arxiv:2025}, $\hcl{x} \setminus \{x\}$ is denoted by $NH(x)$ or $NH^{1}(x)$, and the set $\chcl{x}$ coincides with $NH^{\infty}(x) \cup \{x\}$.
\end{subremark}

The following lemma generalizes~\cite[Lemma~2.2(2)]{MaksymenkoPolulyakh:MFAT:2016} to the case of a non-Hausdorff space $\Bone$:
\begin{sublemma}\label{lm:f_hcl_to_hcl}
Let $f\colon \Aone\to\Bone$ be a continuous map between topological spaces.
\begin{enumerate}[label*={\rm(\alph*)}, leftmargin=*]
\item\label{enum:lm:f_hcl_to_hcl:hcl}
Then $f(\hcl{x}) \subset \hcl{f(x)}$ for all $x\in\Aone$.
In other words, if $x,y \in \Aone$ are inseparable, then their images $f(x)$ and $f(y)$ are also inseparable.

\item\label{enum:lm:f_hcl_to_hcl:chcl}
Moreover, $f(\chcl{x}) \subset \chcl{f(x)}$ for each point $x\in\Aone$.
In other words, $f$ sends $\cbr$-equivalence classes on $\Aone$ to $\cbr$-equivalence classes on $\Bone$, and therefore induces a continuous map $\bar{f}\colon\ChNSQ[\Aone]\to\ChNSQ[\Bone]$ such that $\apr[\Bone]\circ f = \bar{f}\circ\apr[\Aone]$, i.e.\ the following diagram commutes:
\[
\xymatrix@C=5em{
    \Aone \ar[r]^{f} \ar[d]_{\apr[\Aone]} & \Bone \ar[d]^{\apr[\Bone]} \\
    \ChNSQ[\Aone] \ar[r]^{\bar{f}} & \ChNSQ[\Bone]
}
\]

\item\label{enum:lm:f_hcl_to_hcl:hausd}
If $\Bone$ is Hausdorff, then $\apr[\Bone]$ is a homeomorphism, and $f$ is constant on $\cbr$-equivalence classes.
Therefore there exists a unique continuous map $\hat{f} = \apr[\Bone]^{-1}\circ \bar{f}\colon\ChNSQ[\Aone]\to\Bone$ such that $f = \hat{f}\circ \apr[\Aone]$, i.e.\ the following diagram commutes:
\[
\xymatrix@C=5em{
    \Aone \ar[r]^{f} \ar[d]_{\apr[\Aone]} & \Bone \ar[d]^{\apr[\Bone]}_{\cong} \\
    \ChNSQ[\Aone] \ar[ur]^{\hat{f}} \ar[r]_{\bar{f}} & \ChNSQ[\Bone]
}
\]
\end{enumerate}
\end{sublemma}
\begin{proof}
\ref{enum:lm:f_hcl_to_hcl:hcl}
Suppose there exist neighbourhoods $\Vman_1$ and $\Vman_2$ of the points $f(x)$ and $f(y)$ respectively such that $\Vman_1\cap\Vman_2=\varnothing$.
Then their preimages $f^{-1}(\Vman_1)$ and $f^{-1}(\Vman_2)$ are disjoint neighbourhoods of the points $x$ and $y$, which contradicts the inseparability of $x$ and $y$.

\ref{enum:lm:f_hcl_to_hcl:chcl}
Let $x\cbr y\in\Aone$, i.e.\ there exists a finite sequence of points $x=x_0, x_1,\ldots,x_k = y$ such that $x_i\in\hcl{x_{i+1}}$ for all $i=0,\ldots,k-1$.
Then, by~\ref{enum:lm:f_hcl_to_hcl:hcl}, $f(x_i) \in \hcl{f(x_{i+1})}$, and hence $f(x) = f(x_0) \cbr f(x_{k}) = f(y)$.
In other words, $f(y) \in \chcl{f(x)}$, so $\bar{f}$ is well-defined.
Continuity of $\bar{f}$ follows from Lemma~\ref{lm:qmaps}\ref{enum:lm:qmaps:cont_compos}.

Statement~\ref{enum:lm:f_hcl_to_hcl:hausd} follows directly from~\ref{enum:lm:f_hcl_to_hcl:chcl}.
\end{proof}

The following lemma contains some elementary statements from~\cite[Lemma~2.2(3)]{MaksymenkoPolulyakh:MFAT:2016}, \cite[Lemma~2.2.2 and Corollary~2.2.3]{MaksymenkoPolulyakh:PIGC:2021}.
\begin{sublemma}\label{lm:ahaus}
Let $\Aone$ be a topological space and $\apr\colon\Aone\to\Xone$ the quotient map onto the space of its chain inseparability classes.
\begin{enumerate}[leftmargin=*, label={\rm(\arabic*)}]
\item\label{enum:lm:ahaus:saturation}
For an arbitrary subset $\Bman\subset\Aone$, its saturation
\begin{equation}\label{equ:piinv_pi_B}
    \chcl{\Bman} = \chcl{\Bman\cap\Abr} \sqcup (\Bman\cap\Ahaus).
\end{equation}
In particular, if $\Bman\subseteq\Ahaus$ or $\Abr\subseteq\Bman$, then $\Bman$ is $\cbr$-saturated.

\item\label{enum:lm:ahaus:discr}
Let $\Bman\subset\Aone$ be an arbitrary subset, $\beta = \{ \chcl{x} \mid x \in \Bman\}$ the family of chain inseparability classes of points from $\Bman$, $x\in\Bman$ a point, and $y= \apr(x)$ its image in $\Xone$.
Suppose that $\Abr \subseteq \Bman$.
\begin{enumerate}[label={\rm(\alph*)}, leftmargin=*]
\item\label{enum:lm:ahaus:discr:isol_piB}
Then for every open isolating neighbourhood $\Vman\subset\Xone$ of $y$ with respect to $\apr(\Bman)$, its preimage $\apr^{-1}(\Vman)$ is an isolating neighbourhood of the class $\chcl{x}$ with respect to the family $\beta$.
\item\label{enum:lm:ahaus:discr:isol_B}
Conversely, every open neighbourhood $\Uman$ of the set $\chcl{x}$ that is isolating for it with respect to the family $\beta$ is $\cbr$-saturated, and its image $\apr(\Uman)$ is an open isolating neighbourhood of the point $y$ with respect to $\apr(\Bman)$.
\end{enumerate}
In particular, $\apr(\Bman)$ is a discrete subset of $\Xone$ if and only if $\beta = \{ \chcl{x} \mid x \in \Bman\}$ is a discrete family in $\Aone$.

\item\label{enum:lm:ahaus:one_pt_in_hclx}
Let $\Aman\subset\Aone$ be a subset such that $\Aman\cap\hcl{x} = \{x\}$ for each point $x\in\Aman$.
Then $\Aman$ is Hausdorff in the induced topology.

\item\label{enum:lm:ahaus:emb}
The set $\Ahaus$ is a Hausdorff space in the induced topology.
If $\Ahaus$ is open or closed, then so is its image $\apr(\Ahaus)$, and the restriction $\restr{\apr}{\Ahaus}\colon \Ahaus\to\Xone$ is a topological embedding, i.e.\ a homeomorphism onto its image $\apr(\Ahaus)$.

\item\label{enum:lm:ahaus:compact}
Every compact subset $\Kman\subset\Ahaus$ is closed in $\Aone$.
\end{enumerate}
\end{sublemma}
\begin{proof}
\ref{enum:lm:ahaus:saturation}
Since $\Abr$ and $\Ahaus$ are $\cbr$-saturated and $\Aone = \Abr\sqcup\Ahaus$, we have
\[ \chcl{\Bman} = \chcl{\Bman\cap\Abr} \sqcup \chcl{\Bman\cap\Ahaus}. \]
But $\chcl{\Bman\cap\Ahaus}=\Bman\cap\Ahaus$, which proves~\eqref{equ:piinv_pi_B}.

In particular, if $\Bman\subset\Ahaus$, i.e.\ $\Bman\cap\Abr=\varnothing$, then $\chcl{\Bman} = \chcl{\varnothing} \sqcup (\Bman\cap\Ahaus) = \Bman$.

If $\Abr\subset\Bman$, then $\Bman\cap\Abr = \Abr$, and therefore
\[\chcl{\Bman} = \Abr \sqcup (\Bman\cap\Ahaus) = \Abr \sqcup (\Bman\setminus\Abr) = \Bman.\]

\ref{enum:lm:ahaus:discr}
First note that since $\Abr\subset\Bman$, $\Bman$ is $\cbr$-saturated.

\ref{enum:lm:ahaus:discr:isol_piB}
If $\Vman\cap\apr(\Bman) = \{q\}$, then $\Uman=\apr^{-1}(\Vman)$ is an open neighbourhood of the set $\chcl{x}$ and moreover
\[ \apr^{-1}(\Vman)\cap\Bman = \apr^{-1}(q)=\chcl{x}. \]
Thus $\Uman$ is isolating for $\chcl{x}$ with respect to $\beta$.

\ref{enum:lm:ahaus:discr:isol_B}
Conversely, let $\Uman$ be an isolating neighbourhood of $\chcl{x}$ with respect to $\beta$, i.e.\ $\Uman\cap\Bman = \chcl{x}$.
Note that the intersection $\Uman\cap\Ahaus$ is $\cbr$-saturated.
Therefore, to prove that $\Uman$ is $\cbr$-saturated it suffices to show that $\Uman\cap\Abr$ is also $\cbr$-saturated.

Indeed, since $\Abr\subset\Bman$, we have
\[
    \Uman\cap\Abr\subset \Uman\cap\Bman = \chcl{x}.
\]
If $x\in\Ahaus$, then $\Uman\cap\Bman = \chcl{x} = \{x\} \subset \Ahaus$, and therefore $\Uman\cap\Abr=\varnothing$.
If $x\in\Abr$, then $\Uman\cap\Abr= \Uman\cap\Bman = \chcl{x}$.
Thus $\Uman\cap\Abr$ is $\cbr$-saturated.

Now the $\cbr$-saturatedness of $\Uman$ implies that $\Vman:=\apr(\Uman)$ is an open neighbourhood of $y$ in $\Xone$ and contains no points of $\apr(\Bman)$ other than $y$.

\ref{enum:lm:ahaus:one_pt_in_hclx}
For each point $x\in\Aman$ we have $\{x\} \subset \hcl[\Aman]{x} \stackrel{\eqref{equ:hclA_x}}{\subset} \hcl{x} \cap \Aman = \{x\}$.
Therefore $\hcl[\Aman]{x} = \{x\}$ for each $x\in\Aman$, and hence $\Aman$ is Hausdorff in the induced topology.

\ref{enum:lm:ahaus:emb}
The Hausdorffness of $\Ahaus$ follows from~\ref{enum:lm:ahaus:one_pt_in_hclx}, and the injectivity of $\restr{\apr}{\Ahaus}\colon\Ahaus\to\Xone$ follows from the definition of the set $\Ahaus$.
If $\Ahaus$ is open or closed, then its image $\apr(\Ahaus)$ has the same property by the fact that $\apr$ is a quotient map, and in this case, by Lemma~\ref{lm:qmaps}\ref{enum:lm:qmaps:restr_to_inv_img}, $\apr$ induces a homeomorphism from $\Ahaus$ onto $\apr(\Ahaus)$.

Statement~\ref{enum:lm:ahaus:compact} is proved in~\cite[Corollary~2.2.3]{MaksymenkoPolulyakh:PIGC:2021}.
\end{proof}

\subsection{One-dimensional manifolds}
A topological space $(\Aone,\tau)$ is called a \term{one-dimensional manifold} if for each point $x\in\Aone$ one of the following equivalent properties holds:
\begin{itemize}[leftmargin=*]
\item there exists an open neighbourhood $\Uman$ of $x$ homeomorphic either to $(-1,1)$ or to $[0;1)$; such neighbourhoods will be called \term{standard};
\item there exists an open neighbourhood $\Uman$ of $x$ and an open embedding $\Achart\colon\Uman\to[0;1)$; such embeddings are called \term{(coordinate) charts}.
\end{itemize}

The following elementary lemma is the one-dimensional case of the Brouwer invariance of domain theorem.
This statement has a ``homological'' nature, but in the one-dimensional case it reduces to connectedness.
For the reader's convenience, we give a short proof.
\begin{sublemma}\label{lm:brower}
Let $\Aone$ be a one-dimensional manifold, $\Uman\subset\Aone$ an open subset, $x\in\Uman$, and $\Achart,\Bchart\colon\Uman\to[0;1)$ two open embeddings.
Then
\begin{itemize}
\item either $\Achart(x) = \Bchart(x) = 0$, in which case $x$ is called a \term{boundary point},
\item or $\Achart(x), \Bchart(x) \in (0;1)$, in which case $x$ is called an \term{interior point}.
\end{itemize}
\end{sublemma}
\begin{proof}
Note that the point $0\in[0;1)$ is characterised by having a basis of open connected neighbourhoods $\Vman$ such that the complement $\Vman\setminus\{0\}$ is connected.
On the other hand, every point $t\in(0;1)$ has a basis of open connected neighbourhoods $\Vman$ such that $\Vman\setminus\{t\}$ is disconnected.
The statement of the lemma now follows because the above properties are preserved under the homeomorphisms $\Achart$ and $\Bchart$.
\end{proof}

The set of all boundary points of a one-dimensional manifold $\Aone$ is denoted by $\partial\Aone$ and called the \term{boundary} of $\Aone$, and its complement $\Int{\Aone}:=\Aone\setminus\partial\Aone$ is called the \term{interior} of $\Aone$.
It follows easily from the definition that $\Int{\Aone}$ is open, while $\partial\Aone$ is closed and discrete.
\begin{corollary}\label{cor:open_emb_1manif}
Let $\Aone$ and $\Bone$ be one-dimensional manifolds and $\Achart\colon\Aone\to\Bone$ an open embedding.
Then $\Achart(\Int{\Aone}) \subset \Int{\Bone}$ and $\Achart(\partial\Aone) \subset \partial\Bone$.
\end{corollary}

\subsection{(Open) one-dimensional CW complexes}\label{sect:1CW}
For $a<b\in\bR$, each of the following subsets: $(a;b)$ $[a;b)$, $(a;b]$, $[a;b] \subset \bR$ will be called a \term{segment} with endpoints $a$ and $b$.
Here, to simplify some arguments, $(b;a)$ with $b>a$ denotes the interval $(a;b)$, and analogous conventions are used for $[b;a]$, $[b;a)$ and $(b;a]$.
In particular, if $i\in\{0,1\} = \partial[0;1]$ and $\eps\in(0;1)$ is some number, then $[i;\eps) \equiv (\eps;i]$ is always a neighbourhood of the point $i$ in $[0;1]$.

Let $\Lambda$ be a discrete topological space.
Then the topological product $\CEman = [0;1]\times\Lambda$ is a one-dimensional manifold.
Note that every (possibly empty) subset $\EmptyClass \subset \{0;1\}\times\Lambda$ is closed in $\CEman$, and therefore $\Eman = \CEman\setminus\EmptyClass$ is also a one-dimensional manifold, each connected component $\Kman_{\lambda}=\Eman\cap\bigl([0;1]\times\{\lambda\}\bigr)$ of which is a segment.
Here
\begin{align*}
    \partial\Eman &:= \Eman \cap \bigl(\{0,1\} \times \Lambda\bigr), &
    \Int{\Eman}   &:= (0;1)\times\Lambda.
\end{align*}

Conversely, it is easy to see that every one-dimensional manifold each of whose connected components is a segment is homeomorphic to a space of the form $([0;1]\times\Lambda)\setminus\EmptyClass$ for some discrete space $\Lambda$ and some subset $\EmptyClass\subset\{0;1\}\times\Lambda$.

Fix an arbitrary equivalence relation $\relat$ on the boundary $\partial\Eman$.
Extend it to the whole space $\Eman$ by assuming that the class of each point $t\in\Int{\Eman}$ consists only of that point itself.
Let further $\Acw = \Erelat$ be the set of equivalence classes endowed with the quotient topology with respect to the quotient map $\epr\colon\Eman\to\Acw$.
In other words, $\Acw$ is obtained from the disjoint union of segments by gluing some of their endpoints.
We will call $\Acw$ an \term{open one-dimensional CW complex}, or an \term{open topological graph}.
In this context,
\begin{itemize}
\item
the classes of points from $\partial\Eman$ are called \term{$0$-cells} or \term{vertices} of $\Acw$;
\item
the images of the open intervals $(0;1)\times\{\lambda\}$ are called \term{$1$-cells} or \term{open edges} of $\Acw$;
\item
and the quotient map $\epr\colon\Eman\to\Acw$ is called an \term{atlas}.
\end{itemize}

\begin{subremark}\rm
If $\EmptyClass=\varnothing$, then $\Acw$ is a \term{one-dimensional CW complex} in the usual sense.
On the other hand, suppose $\EmptyClass\ne\varnothing$.
Extend the relation $\relat$ to $\partial\CEman$ by adding all points of $\EmptyClass$ as a separate class, and construct the one-dimensional CW complex $\CErelat$ with respect to this relation.
Let also $\epr\colon\CEman\to\CErelat$ be the corresponding quotient map.
Then $\epr(\EmptyClass)$ is a single vertex in $\CErelat$, and $\Acw = (\CErelat) \setminus \epr(\EmptyClass)$.
In other words, an open one-dimensional CW complex is
\begin{itemize}
\item either an ordinary one-dimensional CW complex (the case $\EmptyClass=\varnothing$),
\item or the complement of a single vertex in some one-dimensional CW complex (the case $\EmptyClass\ne\varnothing$).
\end{itemize}
\end{subremark}

A \term{structure of an (open) topological graph} on a topological space $\Xone$ is a homeomorphism $\xi\colon\Acw\to\Xone$ of some (open) topological graph $\Acw$ onto $\Xone$.
In this case, if $\epr\colon\Eman\to\Acw$ is an atlas for $\Acw$, then the composition $\xi\circ\epr\colon\Eman\to\Xone$ is also called an \term{atlas} for $\Xone$, and the images of the vertices (resp.\ open edges) of $\Acw$ are called \term{vertices} (resp.\ \term{open edges}) of $\Xone$.

The following simple lemma gives a characterisation of atlases.
\begin{sublemma}\label{lm:char:1cw}
Let $\Lambda$ be a discrete topological space, $\CEman = [0;1]\times \Lambda$, $\EmptyClass\subset\partial\CEman$ an arbitrary subset, $\Eman = \CEman\setminus\EmptyClass$, and $\func\colon\Eman\to\Xone$ a continuous surjective map onto some topological space $\Xone$.
Then the following conditions are equivalent:
\begin{enumerate}[label={\rm(\arabic*)}, leftmargin=*]
\item\label{enum:lm:char:1cw:struct}
$\func$ defines a structure of a topological graph on $\Xone$, i.e.\ there exists an equivalence relation $\relat$ on $\partial\Eman$ and a homeomorphism $\xi\colon\Erelat \to \Xone$ such that $\func = \xi\circ\epr$, where $\epr\colon\Eman\to\Erelat$ is the quotient map.

\item\label{enum:lm:char:1cw:quotient}
$\func$ is a quotient map, $\func(\partial\Eman) \cap \func(\Int{\Eman}) = \varnothing$, and the restriction $\restr{\func}{\Int{\Eman}}\colon\Int{\Eman} \to \Xone$ is injective.
\end{enumerate}
\end{sublemma}
\begin{proof}
The implication \ref{enum:lm:char:1cw:struct}$\Rightarrow$\ref{enum:lm:char:1cw:quotient} is obvious.

\ref{enum:lm:char:1cw:quotient}$\Rightarrow$\ref{enum:lm:char:1cw:struct}.
Introduce the following equivalence relation $\relat$ on $\partial\Eman$: for $a,b\in\partial\Eman$, $a\relat b$ if and only if $\func(a) = \func(b)$.
Extend it to $\Int{\Eman}$ by setting $x\relat y$ if and only if $x=y$.
Let $\Erelat$ be the corresponding quotient space.
The condition $\func(\partial\Eman) \cap \func(\Int{\Eman}) = \varnothing$ and the injectivity of the restriction $\restr{\func}{\Int{\Eman}}$ show that $x\relat y\in\Eman$ if and only if $\func(x) = \func(y)$.
Therefore there exists a unique bijection $\xi\colon\Erelat\to\Xone$ such that $\func = \xi\circ\epr$, and hence $\epr = \xi^{-1}\circ\func$.
Since $\func$ and $\epr$ are quotient maps, by Lemma~\ref{lm:qmaps}\ref{enum:lm:qmaps:cont_compos}, $\xi$ and $\xi^{-1}$ are continuous, i.e.\ $\xi$ is a homeomorphism.
\end{proof}

\begingroup
\newcommand\II[1]{I_{#1}}
We make one more simple observation about the structure of neighbourhoods of vertices of a topological graph $\Acw$.
\begin{sublemma}\label{lm:1cw_vert_nbh}
\begin{enumerate}[label={\rm(\arabic*)}, leftmargin=*]
\item\label{enum:lm:1cw_vert_nbh:nbh}
Let $x$ be an arbitrary vertex of $\Acw$ and $\epr^{-1}(x) = \{ (t_i, \lambda_i) \in \partial\Eman \}_{i\in\II{x}}$ the corresponding equivalence class of $\relat$, where $t_i \in \{0;1\}$ and $\lambda_i \in \Lambda$ for all $i\in\II{x}$.
For each $i\in\II{x}$ choose an arbitrary number $\eps_i\in(0;1)$ and consider the sets
\begin{align}\label{equ:Ux_Vx}
    \Uman_x &= \mathop{\cup}\limits_{i\in\II{x}}( [t_i;\eps_i) \times \{\lambda_i\} ) \ \subset \Eman, &
    \Vman_x &= \epr(\Uman_x) \subset \Acw.
\end{align}
Then $\Vman_x$ is an open neighbourhood of the point $x\in\Acw$ that contains no other vertices of $\Acw$, and also no open edge of $\Acw$.

\item\label{enum:lm:1cw_vert_nbh:int_bd}
The restriction $\epr\colon\Int{\Eman}\to\Acw$ is an open embedding, and the set of vertices $\epr(\partial\Eman)$ is closed and strongly discrete,

\item\label{enum:lm:1cw_vert_nbh:haus}
$\Acw$ is Hausdorff.
\end{enumerate}
\end{sublemma}
\begin{proof}
\ref{enum:lm:1cw_vert_nbh:nbh}
Note that each set $[t_i;\eps_i) \times \{\lambda_i\}$ is open in $\Eman \cap \bigl([0;1]\times\{\lambda_{i}\}\bigr)$, being a product of the open sets $[t_i;\eps_i)$ of the interval $[0;1]$ and $\{\lambda_i\}$ of the discrete space $\Lambda$.
Therefore their union $\Uman_x$ is an open subset of $\Eman$.
Moreover, it is saturated with respect to $\relat$, i.e.\ $\Uman_x = \epr^{-1}( \epr(\Uman_x)) = \epr^{-1}(\Vman_x)$.
Therefore, since $\epr$ is a quotient map, $\Vman_x$ is also open.

\ref{enum:lm:1cw_vert_nbh:int_bd}
Note that $\Int{\Eman} = \epr^{-1}\bigl(\epr(\Int{\Eman})\bigr)$ and the restriction $\restr{\epr}{\Int{\Eman}}\colon\Int{\Eman}\to\Acw$ is injective.
Therefore, by Lemma~\ref{lm:qmaps}\ref{enum:lm:qmaps:restr_to_inv_img}, it is a homeomorphism onto its image.
Since $\Int{\Eman}$ is open in $\Eman$, its image $\epr(\Int{\Eman})$ is open, i.e.\ $\restr{\epr}{\Int{\Eman}}$ is an open embedding.

It follows that the set of vertices $\epr(\partial\Eman) = \Acw \setminus \epr(\Int{\Eman})$ is closed in $\Acw$.

It remains to prove that the set of vertices $\epr(\partial\Eman)$ is also strongly discrete.
For each $\lambda\in\Lambda$ choose an arbitrary $\eps_{\lambda} \in (0;1)$; for example, one can put $\eps_{\lambda}=\tfrac{1}{2}$.
Let
\begin{equation}\label{equ:Ux_def}
    \Uman_x = \mathop{\cup}\limits_{i\in\II{x}}( [t_i;\eps_{\lambda_i}) \times \{\lambda_i\} )
\end{equation}
be the corresponding neighbourhood of the vertex $x\in\Acw$.
We claim that
\begin{equation}\label{equ:Ux_cap_Uy_empty}
    \Uman_x\cap \Uman_y=\varnothing \ \text{for distinct vertices} \ x,y\in\Acw.
\end{equation}
Since $\Uman_x$ and $\Uman_y$ are saturated with respect to $\relat$, it will follow from~\eqref{equ:Ux_cap_Uy_empty} that $\Vman_x\cap\Vman_y=\varnothing$.

So suppose that $\Uman_x\cap \Uman_y$ contains some point $(s,\lambda)\in\Eman$, i.e.
\[
    \Uman_x \ \supset \ [t_{x};\eps_{\lambda}) \times\{\lambda\}
    \ \ni \ (s,\lambda) \  \in \
    [t_{y};\eps_{\lambda}) \times\{\lambda\} \ \subset \ \Uman_{y}
\]
for some $t_{x},t_{y}\in\{0,1\}$.
Since $\epr(t_{x},\lambda) = x \ne y=\epr(t_{y},\lambda)$, we have $t_{x}\ne t_{y}$.
Therefore we may assume that $t_{x}=0$ and $t_{y}=1$.
But then $s \in [0;\eps_{\lambda}) \cap [1;\eps_{\lambda})=\varnothing$, which is impossible.
Thus $\Uman_x\cap \Uman_y=\varnothing$ for $x\ne y$.

\ref{enum:lm:1cw_vert_nbh:haus}
We need to show that every two points of $\Acw$ are separable.
Consider three cases.

a) Strong discreteness of the set of vertices shows that every two vertices of $\Acw$ are pairwise separable.

b) Next, by~\ref{enum:lm:1cw_vert_nbh:int_bd}, the union $\epr(\Int{\Eman})$ of all open edges of $\Acw$ is open and Hausdorff (homeomorphic to the Hausdorff space $\Int{\Eman} = (0;1) \times \Lambda$).
Therefore every two points of $\epr(\Int{\Eman})$ are also pairwise separable in $\Acw$.

c) Finally, let $x\in\Acw$ be an arbitrary vertex of $\Acw$ and $y = \epr(s,\lambda)$ an arbitrary point of an open edge of $\Acw$, where $(s,\lambda) \in \Int{\Eman}$.
Choose an arbitrary $\delta>0$ such that $[s-\delta;s+\delta] \subset (0;1)$.
On the other hand, let $\epr^{-1}(x) = \{ (t_i, \lambda_i) \in \partial\Eman \}_{i\in\II{x}}$ be the equivalence class of $\relat$ corresponding to the point $x$.
Choose $\eps_i\in(0;1)$, $i\in\II{x}$, such that the corresponding set $\Uman_x = \mathop{\cup}\limits_{i\in\II{x}} [t_i;\eps_i) \times \{\lambda_i\}$ does not intersect $(s-\delta;s+\delta)\times\{\lambda\}$.
Then $\Vman_x := \epr(\Uman_x)$ and $\Vman_y:=\epr\bigl((s-\delta;s+\delta)\times\{\lambda\}\bigr)$ are open disjoint neighbourhoods of the points $x$ and $y$ in $\Acw$.
\end{proof}
\endgroup

\section{Open embeddings of intervals}\label{sect:open_embeddings_01}
In this section, we present several technical statements concerning open embeddings of intervals.

\subsection{Open embeddings of intervals}\label{sect:open_embeddings_of_intervals}

Let $\Aman,\Aone$ be topological spaces.
\begin{subdefinition}[{\rm\cite[Problem 103, p. 117]{Wilansky:TFA:1970}}]
\label{def:oks_embeddings}
A continuous map $\Achart\colon\Aman\to\Aone$ is called%
\footnote{In~\cite{Williams:PJM:1973} KC"=maps are also called \term{semi-closed}, but in modern works this term is used in a different sense.}
a \term{KC"=map (compact closed)} if for every compact subset $\Kman\subset\Aman$ its image $\Achart(\Kman)$ is closed in $\Aone$.
If $\Achart$ is additionally an open embedding, i.e.\ a homeomorphism onto an open subset $\Achart(\Aman) \subset \Aone$, then $\Achart$ will be called an \term{OKC"=embedding}.
\end{subdefinition}

\begin{subexample}\rm
If $\Aone$ is Hausdorff, then every map $\Achart\colon\Aman\to\Aone$ is a KC"=map.
More generally, by Lemma~\ref{lm:ahaus}\ref{enum:lm:ahaus:compact}, every map $\Achart\colon\Aman\to\Aone$ whose image lies in the set $\Ahaus$ of Hausdorff points of $\Aone$ is also a KC"=map.
\end{subexample}

Let $\Kman_{a,b}$ be a segment with endpoints $a<b\in\bR$, i.e.\ $(a;b) \subseteq \Kman_{a,b} \subseteq [a;b]$, and let $\Achart\colon\Kman_{a,b}\to\Aone$ be a continuous map.
For each $t\in(a;b)$ the sets $\Aman_{t} = \Achart\bigl((a;t]\bigr)$ and $\Bman_{t} = \Achart\bigl([t;b)\bigr)$ will be called the \term{left} and \term{right $t$"=edges of $\Achart$}, and the sets $\lms[t]{a}:=\Cl{\Aman_{t}}\setminus\Aman_{t}$ and $\lms[t]{b}:=\Cl{\Bman_{t}}\setminus\Bman_{t}$ the \term{left} and \term{right $t$"=limit sets of $\Achart$}.
Note that these sets depend only on the restriction $\restr{\Achart}{(a;b)}$, but later it will be convenient to consider such sets also for maps of segments, not only of open intervals.

We are interested in conditions under which an OKC"=embedding of an open interval $\Achart\colon(a;b)\to\Aone$ extends continuously to the point $a$ or $b$, see~\cite{FoxArtin:AM:1948}.

\begin{subexample}
\rm Consider the following subsets of $\bR^{2}$:
\begin{align*}
    \Gamma &= \bigl\{ (x,\sin\tfrac{1}{x}) \mid x\in(0;1] \bigr\},&
    \Kman  &= 0\times [-1;1],&
    \Aone  &= \Gamma\cup \Kman.
\end{align*}
Then $\Achart\colon(0;1)\to\Aone$, $\Achart(x) = (x,\sin\tfrac{1}{x})$, is an open embedding, all left limit sets $\alpha_t$ coincide with $\Kman$, and all right limit sets $\beta_t$ coincide with the point $(1,\sin(1))\in \Gamma$.
In particular, $\Achart$ extends continuously to the point $1$ by the formula $\Achart(1) = (1,\sin(1))$, but has no continuous extension to the point $0$.
Note also that \emph{for every point $(0,t)\in\Kman$ and all sufficiently small neighbourhoods $\Wman$ of it in $\bR^{2}$, the intersection $\Wman\cap\Achart\bigl((0;1)\bigr)$ is disconnected.}
Lemma~\ref{lm:okc_emb} below generalizes Lemma 2.4 from~\cite{MaksymenkoPolulyakh:MFAT:2016} and shows (see statement~\ref{enum:lm:okc_emb:cont_ext}) that the latter property is precisely the obstruction to a continuous extension of $\Achart$ to the point $0$.
\end{subexample}

\begin{sublemma}[{\rm see~\cite[Lemma 2.4]{MaksymenkoPolulyakh:MFAT:2016}}]
\label{lm:ok_emb}
Let $\Aone$ be a topological space, $\Achart\colon(0;1) \to \Aone$ an open embedding and $\Acomp = \Achart\bigl((0;1)\bigr)$ its image.
Then for every $s\in(0;1)$
\begin{equation}\label{equ:clJ_d0_J_d1}
    \Cl{\Acomp}\setminus\Acomp = \lms[s]{0} \cup \lms[s]{1}.
\end{equation}
\end{sublemma}
\begin{proof}
Since $\Achart$ is an open embedding, the set $\Achart\bigl((s;1)\bigr)$ is open in $\Aone$ and does not intersect $\Aman_s = \Achart\bigl((0;s]\bigr)$.
Therefore it also does not intersect the closure $\Cl{\Aman_s}$:
\begin{equation}\label{equ:ClAs_Achart_s1_empty}
    \Cl{\Aman_s} \cap \Achart\bigl((s;1)\bigr) = \varnothing.
\end{equation}
Hence
\[
    \Cl{\Aman_s} \setminus \Acomp =
    \Cl{\Aman_s} \setminus
        \Bigl(
            \underbrace{\Achart\bigl((0;s]\bigr)}_{\Aman_s}
                \ \sqcup \
            \Achart\bigl((s;1)\bigr)
        \Bigr) =
     \Cl{\Aman_s} \setminus \Aman_s = \lms[s]{0},
\]
and, similarly, $\Cl{\Bman_s} \setminus \Acomp = \lms[s]{1}$.
Since $\Acomp = \Aman_s \cup \Bman_s$, we have $\Cl{\Acomp} = \Cl{\Aman_s} \cup \Cl{\Bman_s}$, and therefore
\[
    \Cl{\Acomp}\setminus\Acomp =
    (\Cl{\Aman_s}\setminus\Acomp) \cup (\Cl{\Bman_s}\setminus\Acomp)
    =
    \lms[s]{0} \cup \lms[s]{1}.
    \qedhere
\]
\end{proof}

\begin{sublemma}\label{lm:okc_emb}
Let $\Aone$ be a topological space, $\Achart\colon(0;1) \to \Aone$ an OKC"=embedding and $\Acomp = \Achart\bigl((0;1)\bigr)$ its image.
For $s<t\in[0;1]$ denote $\AcompImg{s}{t} := \Achart\bigl((s;t)\bigr)$ and $\AcompCImg{s}{t} := \Achart\bigl([s;t]\bigr)$.
\begin{enumerate}[label={\rm(\arabic*)}, leftmargin=*]
\item\label{enum:lm:okc_emb:same_lim_sets}
Then for all $s,t \in (0;1)$
\begin{align}\label{equ:clAs_As}
    \lms[s]{0} &= \Cl{\Aman_{s}} \setminus \Aman_{s}
               = \Cl{\Aman_{t}} \setminus \Aman_{t}
               = \lms[t]{0}, &
    \lms[s]{1} &= \Cl{\Bman_{s}} \setminus \Bman_{s}
               = \Cl{\Bman_{t}} \setminus \Bman_{t}
               = \lms[t]{1},
\end{align}
i.e.\ its $t$"=limit sets $\lms[t]{0}$ and $\lms[t]{1}$ are independent of $t\in(0;1)$.
Denote them by $\lms{0}=\lms{0}(\Achart)$ and $\lms{1}=\lms{1}(\Achart)$, respectively, and call them the \term{left} and \term{right} limit sets of $\Achart$, respectively.

\item\label{enum:lm:okc_emb:intersect_connected}
Suppose there exist a point $x\in\Cl{\Acomp}\setminus\Acomp$, its open neighbourhood $\Vman$, and an open connected component $\Qman$ of $\Vman\setminus\{x\}$ such that $x\in\Cl{\Qman}$ and $\Qman\subset\Acomp$.
Then $\Qman$ coincides with one of the sets $\Acomp=\AcompImg{0}{1}$ or $\Acomp=\AcompImg{0}{t}$ or $\AcompImg{t}{1}$ for some $t\in(0;1)$.

\item\label{enum:lm:okc_emb:cont_ext}
Let $x\in\Aone$.
Then the following conditions are equivalent:
\begin{enumerate}[leftmargin=*, label={\rm(\alph*)}]
\item\label{enum:lm:okc_emb:cont_ext:ext}
$x$ has a basis of open neighbourhoods $\gamma_{x}$ such that for every $\Vman\in\gamma_x$ there exists $t_{\Vman}\in(0;1)$ with $\phi\bigl((0;t_{\Vman})\bigr) \subset \Vman\cap\Acomp$;
\item\label{enum:lm:okc_emb:cont_ext:conn_nbh}
the extension $\bar{\phi}\colon[0;1) \to \Aone$ of $\phi$ to the point $0 \in [0;1)$ defined by the rule $\bar{\phi}(0) = x$ is continuous.
\end{enumerate}
If these conditions hold, then $\lms{0} \subset \hcl{x}$.
An analogous statement holds for $\lms{1}$.
\end{enumerate}
\end{sublemma}
\begin{proof}
\ref{enum:lm:okc_emb:same_lim_sets}
\newcommand\stHalf{\Hman} 
\newcommand\stComp{\Kman} 
For $s<t \in (0;1)$ consider the sets $\stHalf = \Achart\bigl((s;t]\bigr)$ and $\stComp = \Achart\bigl([s;t]\bigr)$, see Fig.~\ref{fig:lr_edges}.
\begin{figure}[htbp!]
\includegraphics[height=2cm]{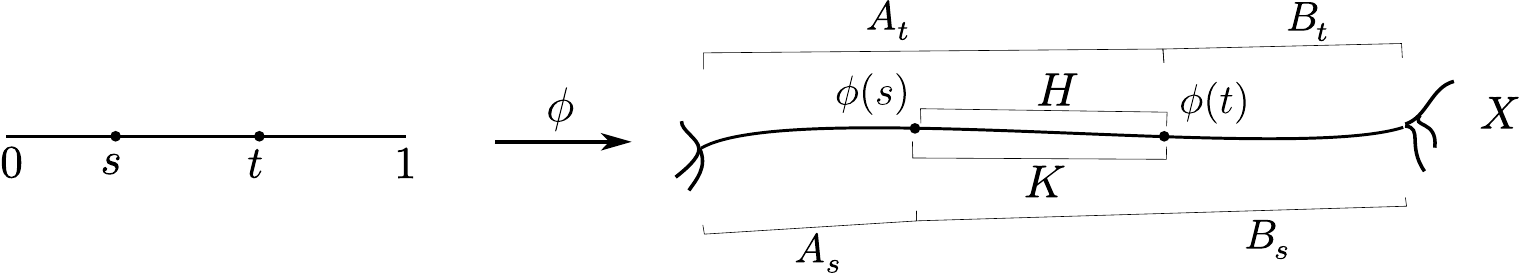}
\caption{Embedding $\Achart$}\label{fig:lr_edges}
\end{figure}

We show that
\begin{equation}\label{equ:clJ_K}
    \Cl{\stHalf} = \stComp.
\end{equation}
Indeed, since $\Achart$ is a KC"=map, the compact set $\stComp$ is closed in $\Aone$, and therefore
\[
    \stComp := \Achart\bigl( [s;t] \bigr)
    \equiv \Achart\bigl( \Cl{(s;t]} \bigr)
    \subset \Cl{ \Achart\bigl( (s;t] \bigr) }
    \equiv \Cl{\stHalf}
    \subset \Cl{\stComp} = \stComp.
\]
Hence
\[
    \Cl{\Aman_t} =
    \Cl{\Aman_s \sqcup \stHalf} =
    \Cl{\Aman_s} \cup \Cl{\stHalf}
        \stackrel{\eqref{equ:clJ_K}}{=}
    \Cl{\Aman_s} \cup \stComp =
    \Cl{\Aman_s} \cup \bigl(\{\Achart(s)\} \sqcup \stHalf\bigr)
        \stackrel{\eqref{equ:ClAs_Achart_s1_empty}}{=}
    \Cl{\Aman_s} \sqcup \stHalf,
\]
and therefore
\[
    \lms[t]{0} = \Cl{\Aman_t}\setminus \Aman_t =
    (\Cl{\Aman_s} \sqcup \stHalf) \setminus (\Aman_s \sqcup \stHalf) =
    \Cl{\Aman_s} \setminus \Aman_s = \lms[s]{0}.
\]
The proof that $\lms[s]{1}=\lms[t]{1}$ is analogous.

\ref{enum:lm:okc_emb:intersect_connected}
Since $\Qman \subset \Acomp$ is connected and open, it has the form $\Qman\equiv\Achart\bigl((s;t)\bigr)=\AcompImg{s}{t}$ for some $s<t\in[0;1]$.
We need to show that $s=0$ or $t=1$.

Indeed, if $[s;t]\subset(0;1)$, then the compact subset $\AcompCImg{s}{t} \subset \Acomp \subset \Aone$ must be closed in $\Aone$.
But then $\Cl{\Qman} = \AcompCImg{s}{t}$, and hence
\[
    x \in
    \Cl{\Qman}\setminus\Qman =
    \AcompCImg{s}{t}\setminus\AcompImg{s}{t} =
    \{s,t\} \subset
    \Acomp,
\]
which contradicts the assumption that $x\notin\Acomp$.

\ref{enum:lm:okc_emb:cont_ext}
Let $x\in\Aone$ and $\bar{\phi}\colon[0;1) \to \Aone$ be the extension of $\phi$ to the point $0 \in [0;1)$ defined by the rule $\bar{\phi}(0) = x$.
Then the equivalence
\ref{enum:lm:okc_emb:cont_ext:ext}
$\Leftrightarrow$
\ref{enum:lm:okc_emb:cont_ext:conn_nbh} is trivial and follows directly from the definition of continuity of $\bar{\phi}$ at $0$.

Assume that condition~\ref{enum:lm:okc_emb:cont_ext:ext} holds.
We show that then $\lms{0}\subset\hcl{x}$.

Indeed, let $y\in\lms{0}$ be an arbitrary point, $\Wman$ an arbitrary open neighbourhood of it, and $\Vman\in\gamma_{x}$ an arbitrary element of the basis $\gamma_{x}$.

Then, by~\ref{enum:lm:okc_emb:cont_ext:ext}, we have $\phi\bigl((0;t)\bigr) \subset \Vman \cap \Acomp$ for some $t\in(0;1)$.
Since
\[
    y \in \lms{0}
        = \lms[t/2]{0}  \subset  \overline{\Aman_{t/2}}
        \subset \overline{\phi\bigl((0;t)\bigr)}
        \subset \overline{\Vman \cap \Acomp}
\]
we have $\Wman\cap\bigl(\Vman \cap \Acomp\bigr) \ne\varnothing$.
In particular, $\Wman\cap\Vman\ne\varnothing$, so $x$ and $y$ are inseparable, i.e.\ $y\in\hcl{x}$.
\end{proof}

\begin{sublemma}\label{lm:Jint}
Let $\Aone$ be a one-dimensional manifold, $x$ an isolated point of the set $\Abr\cup\partial\Aone$, and $\Vman$ a standard neighbourhood of $x$ that is isolating with respect to $\Abr\cup\partial\Aone$.
Let further
\[
    \Achart\colon(0;1)\to\Ahaus\setminus\partial\Aone
\]
be an open embedding such that its image $\Acomp=\Achart\bigl((0;1)\bigr)$ is a connected component of $\Ahaus\setminus\partial\Aone$.
For $s<t\in[0;1]$ denote $\AcompImg{s}{t} := \Achart\bigl((s;t)\bigr)$.

\begin{enumerate}[label={\rm(\arabic*)}, leftmargin=*]
\item\label{enum:lm:Jint:Qcomp}
If some connected component $\Qman$ of $\Vman\setminus\{x\}$ intersects $\Acomp$, then $\Qman$ coincides with one of the sets $\Acomp=\AcompImg{0}{1}$, $\AcompImg{0}{t}$ or $\AcompImg{t}{1}$ for some $t\in[0;1]$.

\item\label{enum:lm:Jint:NWJ:0}
If $x\in\lms{0}(\Achart)$, then there exists $t\in(0;1]$ such that $\AcompImg{0}{t}$ is a connected component of $\Vman\setminus\{x\}$.
Conversely, if $\AcompImg{0}{t}$ is a connected component of $\Vman\setminus\{x\}$ for some $t\in(0;1)$, then $x\in\lms{0}(\Achart)$.

\item\label{enum:lm:Jint:NWJ:1}
Similarly, if $x\in\lms{1}(\Achart)$, then $\AcompImg{t}{1}$ is a connected component of $\Vman\setminus\{x\}$ for some $t\in[0;1)$.
Conversely, if $\AcompImg{t}{1}$ is a connected component of $\Vman\setminus\{x\}$ for some $t\in(0;1)$, then $x\in\lms{1}(\Achart)$.

\item\label{enum:lm:Jint:NWJ:no}
The following conditions are equivalent:
\begin{enumerate}[label={\rm(\alph*)}]
\item $x\notin\lms{0}(\Achart)\cup\lms{1}(\Achart)$;
\item $x\notin\Cl{\Acomp}$;
\item $\Vman\cap\Acomp=\varnothing$.
\end{enumerate}

\item\label{enum:lm:Jint:ext}
If $x\in\lms{i}(\Achart)$ for some $i\in\{0,1\}$, then $\lms{i}(\Achart) \subset \hcl{x}$, and $\Achart$ extends continuously to the point $i\in[0;1]$ by the formula $\Achart(i) = x$.
\end{enumerate}
\end{sublemma}
\begin{proof}
First note that since the image of $\Achart$ lies in $\Ahaus$, it follows from Lemma~\ref{lm:ahaus}\ref{enum:lm:ahaus:compact} that $\Achart$ is an OKC"=map.
Therefore its left and right limit sets $\lms{0}(\Achart)$ and $\lms{1}(\Achart)$ are well-defined, and we will denote them by $\lms{0}$ and $\lms{1}$, respectively.

Recall also that ``standard'' means that $\Vman$ is homeomorphic either to $[0;1)$ or to $(-1;1)$.
Therefore $\Vman\setminus\{x\}$ has, respectively, one or two connected components, each of which is homeomorphic to an open interval and contains $x$ in its closure.

On the other hand, the condition that $\Vman$ is ``isolating'' for $x$ with respect to the set $\Abr\cup\partial\Aone$ means that $\Vman\cap(\Abr\cup\partial\Aone)=\{x\}$.
Therefore each connected component $\Qman$ of $\Vman\setminus\{x\}$ must lie in some connected component of $\Ahaus\setminus\partial\Aone = \Aone\setminus(\Abr\cup\partial\Aone)$.

\ref{enum:lm:Jint:Qcomp}
Since $\Qman\cap\Acomp\ne\varnothing$, we have $\Qman\subset\Acomp$.
Now our statement follows from Lemma~\ref{lm:okc_emb}\ref{enum:lm:okc_emb:intersect_connected}.

\ref{enum:lm:Jint:NWJ:0}.
Assume that $x\in\lms{0}$.
Since
\[
    \lms{0} = \lms[\tau]{0}(\Achart) = \Cl{\Achart\bigl((0;\tau]\bigr)}\setminus\Achart\bigl((0;\tau]\bigr)
\]
for all $\tau\in(0;1)$, the point $x$ belongs to the closure of each set $\Achart\bigl((0;\tau]\bigr)\subset\Acomp$ for all $\tau>0$, and therefore $\Vman$ intersects all such sets.
Then from~\ref{enum:lm:Jint:Qcomp} it follows that $\Vman\setminus\{x\}$ has a connected component of the form $\AcompImg{0}{t}$ for some $t\in(0;1]$.

Conversely, suppose that for some $t\in(0;1)$ the set $\Qman := \AcompImg{0}{t}$ is a connected component of $\Vman\setminus\{x\}$.
We show that then $x\in\lms{0} = \lms[\tau]{0}(\Achart)$ for some $\tau\in(0;1]$.

First suppose that $\Vman$ is homeomorphic to an open interval and let $\Qman_0, \Qman_1$ be the connected components of $\Vman\setminus\{x\}$.
Let further $\Vman'$ be an arbitrary neighbourhood of $x$ also homeomorphic to an open interval.
Then the intersections $\Qman'_i:=\Vman' \cap \Qman_i$, $i=0,1$, are the connected components of $\Vman'\setminus\{x\}$.
In particular, $\Qman' \subset \Qman \subset \Acomp$, and therefore $\Qman' = \AcompImg{s'}{t'} \subset \AcompImg{0}{t}$ for some $s',t'$.
By the previous, either $s'=0$ or $t'=1$.
But $t'\leq t <1$, so $s'=0$, i.e.\ $\Qman' = \AcompImg{0}{t'}$ and hence it intersects $\Achart\bigl((0;\tau]\bigr)$.
Therefore $x\in\Cl{\Achart\bigl((0;\tau]\bigr)}\setminus\Achart\bigl((0;\tau]\bigr) = \lms[\tau]{0}(\Achart) = \lms{0}$.

The case when $\Vman$ is homeomorphic to $[0;1)$ is similar.
This proves~\ref{enum:lm:Jint:NWJ:0}.

Statement~\ref{enum:lm:Jint:NWJ:1} is analogous to~\ref{enum:lm:Jint:NWJ:0}.

\ref{enum:lm:Jint:NWJ:no}
Clearly, (c)$\Rightarrow$(b).
Since $\Cl{\Acomp}\setminus\Acomp = \lms{0}\cup\lms{1}$, see~\eqref{equ:clJ_d0_J_d1}, we have (b)$\Rightarrow$(a).
But $x\notin\Acomp$, therefore (a)$\Rightarrow$(b) as well.

(a)$\Rightarrow$(c)
Assume that (c) does not hold, i.e.\ $\Vman\cap\Acomp\ne\varnothing$.
Then, as shown above, this intersection $\Vman\cap\Acomp$ is one or both components of the complement $\Vman\setminus\{x\}$.
Therefore, by \ref{enum:lm:Jint:NWJ:0} or \ref{enum:lm:Jint:NWJ:1}, $x\in\lms{0}\cup\lms{1}$, i.e.\ (a) is also violated.

\ref{enum:lm:Jint:ext}
For definiteness, assume that $x\in\lms{0}$.
Then by~\ref{enum:lm:Jint:NWJ:0}, for every open neighbourhood $\Vman$ of $x$ its intersection with $\Acomp$ contains $\Achart\bigl((0;t)\bigr)$ for some $t\in(0;1)$.
Therefore the continuity of the extension of $\Achart$ to the point $0\in[0;1)$ by the formula $\Achart(0)=x$ and the property $\lms{0}\subset\hcl{x}$ follow from Lemma~\ref{lm:okc_emb}\ref{enum:lm:okc_emb:cont_ext}.
\end{proof}

\section{Proof of Theorem~\ref{th:1cw_struct}}\label{sect:proof:th:1cw_struct}

Let $\Aone$ be a connected one-dimensional manifold not homeomorphic to the circle $S^1$, $\Abr$ the set of its branch points, $\Ahaus = \Aone\setminus\Abr$ the set of Hausdorff points, and $\apr\colon\Aone\to\Xone$ the quotient map onto the space of \trm{chain inseparability} classes of the points of $\Aone$,
\begin{align*}
\Avrt &= \Abr\cup\partial\Aone,&
\Xvrt &= \apr(\Avrt), \\
\Aedg &= \Aone\setminus\Avrt = \Ahaus\setminus\partial\Aone, &
\Xedg &= \apr(\Aedg) = \Xone\setminus\Xvrt.
\end{align*}

We need to verify the equivalence of the following conditions:
\begin{enumerate}
\item[\ref{enum:th:1cw_struct:loc}]
$\Abr$ is locally finite, and each connected component of $\Aedg$ has a countable base;
\item[\ref{enum:th:1cw_struct:1cw}]
$\Abr$ is nowhere dense, and $\Xone$ has the structure of a topological graph for which $\Xvrt$ is the set of vertices (respectively, the connected components of $\Xedg$ are the open edges of $\Xone$).
\end{enumerate}

\subsection*{Proof of the implication \ref{enum:th:1cw_struct:1cw}$\Rightarrow$\ref{enum:th:1cw_struct:loc}}
Condition~\ref{enum:th:1cw_struct:1cw} means that there exist a discrete topological space $\Lambda$, a subset $\EmptyClass \subset \{0,1\}\times\Lambda$, and a surjective quotient map $\func\colon\OEman = ([0;1]\times\Lambda)\setminus\EmptyClass \to \Xone$ such that $\func(\partial\OEman) = \Xvrt$, $\func(\Int{\OEman}) = \Xedg$, and the restriction map $\restr{\func}{\Int{\OEman}}\colon\Int{\OEman}\to\Xedg$ is a homeomorphism.
In particular, we have the following two diagrams:
\begin{align*}
    &
    \xymatrix{
        & \Avrt \ar[d]^-{\apr} \\
        \partial\OEman \ar[r]^-{\func} & \Xvrt
    }
    &&
    \xymatrix{
        & \Aedg \ar[d]^-{\apr} \\
        \Int{\OEman} = (0;1)\times\Lambda \ar[r]^-{\func} \ar@{-->}[ur] & \Xedg
    }
\end{align*}
where in the left diagram both arrows are surjective, while in the right one $\restr{\func}{\Int{\OEman}}$ is a homeomorphism and $\restr{\apr}{\Aedg}$ is a continuous bijection.

1) Let us show that \term{$\restr{\apr}{\Aedg}$ is also a homeomorphism}.
Then we obtain a homeomorphism
\[
    \apr^{-1}\circ\func \colon \Int{\OEman} = (0;1)\times\Lambda \to \Aedg.
\]
In particular, each connected component of $\Aedg$ will be homeomorphic to an open interval, and hence will have a countable base.

Since $\func$ is a quotient map and $\Int{\OEman} = \func^{-1}(\Xedg)$ is open in $\OEman$, the set $\Xedg$ is open in $\Xone$, and hence $\Aedg = \apr^{-1}(\Xedg)$ is open in $\Aone$.
Then, by Lemma~\ref{lm:qmaps}\ref{enum:lm:qmaps:restr_to_inv_img}, the restriction map $\restr{\apr}{\Aedg}\colon\Aedg\to\Xedg$ is a quotient map.
Since it is a continuous bijection, it is a homeomorphism.

2) Let us prove that \term{$\Abr$ is a locally finite set}.
We show that even $\Avrt = \Abr\cup\partial\Aone$ is locally finite, i.e.\ \term{closed and discrete} (see Lemma~\ref{lm:T1_loc_fin__closed_discr}).

As noted in 1), the set $\Aedg$ is open, and therefore $\Avrt = \Aone\setminus\Aedg$ is closed.
Let us verify the discreteness of $\Avrt$.

2.1)
First, note that by condition~\ref{enum:th:1cw_struct:1cw}, $\Abr$ is nowhere dense.
On the other hand, $\partial\Aone$ is also nowhere dense, and therefore $\Avrt$ is a nowhere dense \term{closed} subset of $\Aone$.

2.2)
Suppose that $\Avrt$ is not discrete, i.e.\ it has a non-isolated point $x$ in the induced topology.
We show that this contradicts the fact that $\Xone$ has the structure of a topological graph compatible with $\cbr$.

Let $\Acomp = (-1;1)$ or $[0;1)$ and $\Achart\colon\Acomp\to\Aone$ be an open embedding such that $\Achart(0)=x$, i.e.\ $\Wman:=\Achart(\Acomp)$ is a standard open neighbourhood of $x$.
Since $x$ is not isolated in $\Avrt$, there exists a sequence of pairwise distinct points $\{x_i\}_{i\in\bN}\subset\Wman\cap(\Avrt\setminus\{x\})$ converging to $x$.
Denote $a_i = \phi^{-1}(x_i)$, $i\in\bN$.
Without loss of generality, we may assume that $\{a_i\}_{i\in\bN}$ is a strictly monotonically decreasing sequence of positive numbers converging to $0$.

Since $\Avrt$ is nowhere dense, each of the intervals $(a_{i+1};a_i)$ must contain a point $b_{i}$ that does not belong to $\Avrt$.
Let $\Bman_i \subset \Wman$ be the connected component of $\Wman\setminus\Avrt=\Wman\cap\Aedg$ containing the point $\Achart(b_i)$.

\begin{subproposition*}
$\Bman_i$ is also a connected component of $\Aedg$.
\end{subproposition*}
\begin{proof}
Indeed, note that $\Wman\setminus\Avrt$ is an \term{open} subset of the locally connected space $\Aone$, and hence its connected component $\Bman_i$ is open in $\Aone$, and therefore also in $\Aedg$.

On the other hand, $\Achart^{-1}(\Bman_i) = (c;d)$ for some $c<d \in (0;1)$.
Since $\Bman_i \subset\Ahaus$, by Lemma~\ref{lm:ahaus}\ref{enum:lm:ahaus:compact} each of its compact subsets is closed in $\Aone$, i.e.\ $\restr{\Achart}{(c;d)}\colon (c;d)\to\Bman_i$ is an OKC"=embedding, see Definition~\ref{def:oks_embeddings}.
Let $\lms{c}=\lms{c}(\restr{\Achart}{(c;d)})$ and $\lms{d}=\lms{d}(\restr{\Achart}{(c;d)})$ be, respectively, the left and right limit sets of this embedding.
Then, by Lemma~\ref{lm:ok_emb}, $\Cl{\Bman_i}\setminus\Bman_i \stackrel{\eqref{equ:clJ_d0_J_d1}}{=} \lms{c} \cup \lms{d}$.
Moreover, since $\restr{\Achart}{(c;d)}$ extends to a continuous map of the closed segment $[c;d]\to\Aone$, by Lemma~\ref{lm:okc_emb}\ref{enum:lm:okc_emb:cont_ext},
\[
    \Cl{\Bman_i}\setminus\Bman_i =
    \lms{c} \cup \lms{d} \subset \hcl{\Achart(c)} \cup \hcl{\Achart(d)} \subset \Avrt.
\]
It follows that $\Bman_i = \Cl{\Bman_i}\cap\Aedg$, and hence $\Bman_i$ is also \term{closed} in $\Aedg$.

Thus $\Bman_i$ is a connected component of $\Aedg$.
\end{proof}

Since the sequence $\{b_i\}_{i\in\bN}$ converges to $0$, as a consequence we obtain that \term{every neighbourhood $\Uman$ of $x$ must contain infinitely many such components $\Bman_i$}.

On the other hand, let $y = \apr(x) \in \Xvrt$ be the corresponding vertex of $\Xone$.
By Lemma~\ref{lm:1cw_vert_nbh}\ref{enum:lm:1cw_vert_nbh:nbh}, see~\eqref{equ:Ux_Vx}, there exists an open neighbourhood $\Vman_{y}$ of $y$ that \term{contains no edge of $\Xedg$}.
This means that the open neighbourhood $\Uman_{x} := \apr^{-1}(\Vman_{y})$ of $x$ contains no connected component of the set $\Aedg = \Aone\setminus\Avrt$.
This contradicts the previous property.

Therefore the closed set $\Avrt$ must also be discrete, and hence it is locally finite.

\subsection*{Proof of the implication \ref{enum:th:1cw_struct:loc}$\Rightarrow$\ref{enum:th:1cw_struct:1cw}}

We need to construct an atlas of a topological graph for $\Xone$ for which $\Xvrt$ is the set of vertices.
The proof consists of a series of claims.

Recall that for arbitrary $a<b\in\bR$ the notations for the intervals $(a;b)$ and $(b;a)$ are identical.

\begin{enumerate}[wide, label={\rm{A\arabic*)}}, itemsep=1em]
\newcommand\statement[2][\empty]{\item\ifx#1\empty\relax\else\label{#1}\fi\emph{#2}\par}
\statement[enum:Avrt_closed_discr]
{The set $\Avrt = \Abr \cup \partial\Aone$ is locally finite, and hence, by Lemma~\ref{lm:T1_loc_fin__closed_discr}, it is closed and discrete.}
By hypothesis, $\Abr$ is locally finite.
Moreover, $\partial\Aone$ is locally finite by definition.
Therefore $\Avrt = \Abr \cup \partial\Aone$ is also locally finite.

\statement[enum:Aedg_open_dense]
{The set $\Aedg = \Ahaus \setminus \partial\Aone$ is open and everywhere dense, and each of its connected components is homeomorphic to an open interval.}
Openness and everywhere density of $\Aedg = \Aone\setminus\Avrt$ follow from~\ref{enum:Avrt_closed_discr}.

Since $\Aone$ is a locally connected space and $\Aedg$ is open, each connected component $\Acomp$ of $\Aedg$ must also be open.
Moreover, by Lemma~\ref{lm:ahaus}\ref{enum:lm:ahaus:emb}, $\Aedg$, and hence $\Acomp$, are Hausdorff spaces.
Thus $\Acomp$ is a connected Hausdorff one-dimensional manifold with empty boundary, which, by condition~\ref{enum:th:1cw_struct:loc}, also has a countable base.
Therefore $\Acomp$ is homeomorphic either to $S^1$ or to $(0;1)$.

If $\Acomp\cong S^1$, then it is compact and, by Lemma~\ref{lm:ahaus}\ref{enum:lm:ahaus:compact}, it is also closed in $\Aone$.
Since $\Aone$ is connected, $\Aone = \Acomp \cong S^1$, contradicting the assumption that $\Aone\not\cong S^1$.
Thus $\Acomp \cong (0;1)$.

\statement{The set $\Xedg=\apr(\Aedg)$ is open, and $\Xvrt=\apr(\Avrt)$ is closed in $\Xone$.}
Note that $\Aedg$ and $\Avrt$ are $\cbr$"=saturated, i.e.\ $\Aedg = \apr^{-1}(\Xedg)$ and $\Avrt = \apr^{-1}(\Xvrt)$.
Now the openness of $\Xedg$ and the closedness of $\Xvrt$ follow from the fact that $\apr$ is a quotient map, together with the openness of $\Aedg$ and the closedness of $\Avrt$.

\statement[enum:clj_j]{$\Cl{\Acomp}\setminus\Acomp \subset \Avrt = \Aone\setminus\Aedg$ for each connected component $\Acomp$ of $\Aedg$.}
This is a direct consequence of Lemma~\ref{lm:open_disj_family}, since the components of $\Aedg$ are open and pairwise disjoint.

\statement[enum:jext]{For each connected component $\Acomp$ of $\Aedg$ there exists a continuous map $\Achart\colon\Kman\to\Aone$ of some segment $\Kman$ with endpoints $0$ and $1$, i.e.\ $(0;1)\subseteq\Kman\subseteq[0;1]$, such that
\begin{enumerate}[leftmargin=5em, label={\rm(A5-\arabic*)}]
\item\label{enum:jext:int} $\Achart$ homeomorphically maps $\Int{\Kman} = (0;1)$ onto $\Acomp$;
\item\label{enum:jext:bd} $\Achart(\partial\Kman)\subset\Avrt$;
\item\label{enum:jext:lms} $i \in \partial\Kman$ if and only if $\lms{i}(\Achart)\ne\varnothing$, and in this case $\lms{i}(\Achart) \subset \hcl{\Achart(i)}$.
\end{enumerate}
Moreover,
\begin{equation}\label{equ:achart_dK1}
    \Achart(\partial\Kman)
        \ \subset \
    \Cl{\Acomp} \setminus \Acomp
        \ \stackrel{\eqref{equ:clJ_d0_J_d1}}{=} \
        \mathop{\cup}\limits_{i \in \partial\Kman} \lms{i}(\Achart)
        \ \subset \
        \mathop{\cup}\limits_{i \in \partial\Kman} \hcl{\Achart(i)}
        \ \subset \
    \Avrt.
\end{equation}}

Indeed, since the image of $\Achart$ lies in $\Ahaus$, by Lemma~\ref{lm:ahaus}\ref{enum:lm:ahaus:compact} $\Achart$ is an OKC"=embedding.
In particular, its left and right limit sets $\lms{0}(\Achart)$ and $\lms{1}(\Achart)$ are well-defined.
Recall that by~\eqref{equ:clJ_d0_J_d1} and~\ref{enum:clj_j}, $\Cl{\Acomp} \setminus \Acomp = \lms{0}(\Achart) \cup \lms{1}(\Achart) \subset \Avrt$.

Add to $(0;1)$ those points $i\in \{0,1\}$ for which $\lms{i}(\Achart)\ne\varnothing$ and denote the resulting segment by $\Kman$:
\[
    \Kman = (0;1) \cup \{ i \mid \lms{i}(\Achart) \ne\varnothing\}.
\]
Furthermore, for $i\in\{0,1\}\cap\Kman$ choose an arbitrary point $x_i$ in the set $\lms{i}(\Achart)\ne\varnothing$ and extend $\Achart$ at the point $i$ by the rule $\Achart(i) = x_i$.
Then $x_i\in\Avrt$ and, by Lemma~\ref{lm:Jint}\ref{enum:lm:Jint:ext}, the resulting extension $\Achart\colon\Kman\to\Aone$ is continuous, with $\lms{i}(\Achart) \subset \hcl{x_i}$.
The validity of properties~\ref{enum:jext:int}, \ref{enum:jext:bd}, \ref{enum:jext:lms} and of inclusion~\eqref{equ:achart_dK1} for $\Achart$ is obvious.
\end{enumerate}

{\bf Construction of an atlas for $\Xone$.}
Let $\{\Acomp[\lambda]\}_{\lambda\in\Lambda}$ be the set of all distinct connected components of $\Aedg$.
Since $\Aone \not\cong S^1$, by~\ref{enum:jext}, for each connected component $\Acomp[\lambda]$, $\lambda\in\Lambda$, of $\Aedg$ there exists a continuous map $\Achart[\lambda]\colon\Kman_{\lambda}\to\Aone$ of some segment $\Kman_{\lambda}$ with endpoints $0$ and $1$ such that
\begin{enumerate}[label={\rm(\alph*)}]
\item\label{enum:phil:int}
$\Achart[\lambda]$ homeomorphically maps $\Int{\Kman_{\lambda}} = (0;1)$ onto $\Acomp[\lambda]$;
\item\label{enum:phil:bd}
$\Achart[\lambda](\partial\Kman_{\lambda}) \subset \Avrt$;
\item\label{enum:phil:lms}
$i\in\partial\Kman_{\lambda}$ if and only if $\lms{i}(\Achart[\lambda]) \ne \varnothing$, and in this case $\lms{i}(\Achart[\lambda]) \subset \hcl{\Achart[\lambda](i)}$.
\end{enumerate}
Let further $\Eman = \mathop{\sqcup}\limits_{\lambda\in\Lambda} (\Kman_{\lambda} \times \{\lambda\}) \subset [0;1]\times\Lambda$ be the disjoint union of these segments,
\[
    \Achart = \mathop{\sqcup}\limits_{\lambda\in\Lambda} \Achart[\lambda] \colon\Eman\to\Aone,
    \qquad
    \restr{\Achart}{\Kman_{\lambda}} = \Achart[\lambda],
\]
-- the induced map, and $\func = \apr\circ\Achart\colon\Eman\to\Xone$.
Our goal is to prove that $\func$ is the desired atlas of an open one-dimensional CW complex.
For this it suffices to verify the conditions of Lemma~\ref{lm:char:1cw}\ref{enum:lm:char:1cw:quotient}, see statement~\ref{enum:f_is_an_atlas} below.
But first we describe some properties of the map $\func$.

\begin{enumerate}[wide, label={\rm{A\arabic*)}}, itemsep=1em, start=6]
\newcommand\statement[2][\empty]{\item\ifx#1\empty\relax\else\label{#1}\fi\emph{#2}\par}

\statement[enum:func_int]{$\func$ homeomorphically maps $\Int{\Eman}$ onto $\Xedg$.}
By construction, each $\Achart[\lambda]$ homeomorphically maps $\Int{\Kman_{\lambda}} = (0;1)$ onto $\Acomp[\lambda]$.
Therefore $\Achart$ induces a homeomorphism $\Int{\Eman} = \mathop{\sqcup}\limits_{\lambda\in\Lambda}\Int{\Kman_{\lambda}}$ onto $\mathop{\sqcup}\limits_{\lambda\in\Lambda}\Acomp[\lambda] = \Aedg = \Ahaus\setminus\partial\Aone$.
Since, by Lemma~\ref{lm:ahaus}\ref{enum:lm:ahaus:emb}, $\apr$ further homeomorphically maps $\Ahaus$ onto an open subset of $\Xone$, the map $\func = \apr\circ\Achart$ induces a homeomorphism $\Int{\Eman}$ onto $\apr(\Aedg)=\Xedg$.

\statement[enum:brpt_int]
{For each point $x\in\Abr \cap \Int{\Aone}$ there exist an open neighbourhood $\Wman$, a homeomorphism $\Bchart\colon(-1;1)\to\Wman$,
 $\lambda,\mu\in\Lambda$, $i\in\partial\Kman_{\lambda}$,  $j\in\partial\Kman_{\mu}$,  $\lbnd, \mbnd\in(0;1)$ such that
\begin{enumerate}[leftmargin=5em, label={\rm(A7-\arabic*)}]
\item\label{enum:brpt_int:0} $\Bchart(0)=x$;
\item\label{enum:brpt_int:p} $\lComp := \Achart[\lambda]\bigl((i;\lbnd)\bigr) = \Bchart\bigl((-1;0)\bigr) \subset \Acomp[\lambda]$;
\item\label{enum:brpt_int:q} $\mComp := \Achart[\mu]\bigl((j;\mbnd)\bigr) = \Bchart\bigl((0;1)\bigr) \subset \Acomp[\mu]$;
\item\label{enum:brpt_int:ij} $\Achart[\lambda](i) \cbr x \cbr \Achart[\mu](j)$;
\item\label{enum:brpt_int:another}
for arbitrary $\libnd,\mjbnd\in(0;1)$ the set $\Vman = \Achart[\lambda]\bigl((i;\libnd)\bigr) \cup \{x\} \cup \Achart[\mu]\bigl((j;\mjbnd)\bigr)$
is an open neighbourhood of $x$.
\end{enumerate}
}
Indeed, since $\Abr$ is discrete, $x$ has a neighbourhood $\Wman$ homeomorphic to $(-1;1)$ and such that $\Wman\setminus\{x\} \subset \Ahaus \setminus\partial\Aone = \Aedg$.
Note that $\Wman\setminus\{x\}$ has two connected components $\lComp$ and $\mComp$, which are also homeomorphic to open intervals.
Therefore they are contained in some connected components of $\Aedg$, say $\lComp \subset \Acomp[\lambda]$ and $\mComp \subset \Acomp[\mu]$ for some $\lambda,\mu\in\Lambda$.

Fix an arbitrary homeomorphism $\Bchart\colon(-1;1)\to\Wman$ such that $\Bchart(0)=x$, $\Bchart\bigl((-1;0)\bigr)=\lComp$, and $\Bchart\bigl((0;1)\bigr)=\mComp$.
From Lemma~\ref{lm:Jint} it now follows that $x\in\Cl{\Acomp[\lambda]} \cap \Cl{\Acomp[\mu]}$.
Therefore, by~\eqref{equ:clJ_d0_J_d1}, we have $x\in\lms{i}(\Achart[\lambda]) \cap \lms{j}(\Achart[\mu])$ for some $i,j\in\{0,1\}$.
But then, by statements~\ref{enum:lm:Jint:NWJ:0} and~\ref{enum:lm:Jint:NWJ:1} of Lemma~\ref{lm:Jint}, there exist $s,t\in[0;1]$ such that
\begin{align*}
    i &\in\partial\Kman_{\lambda}, &
    \Achart[\lambda]\bigl((i;s)\bigr) &= \lComp, &
    \lms{i}(\Achart[\lambda]) &\subset \hcl{\Achart[\lambda](i)}, \\
    j&\in\partial\Kman_{\mu}, &
    \Achart[\mu]\bigl((j;t)\bigr) &= \mComp, &
    \lms{j}(\Achart[\mu]) &\subset \hcl{\Achart[\mu](j)}.
\end{align*}
In particular, $x\in \lms{i}(\Achart[\lambda]) \cap \lms{j}(\Achart[\mu]) \subset  \hcl{\Achart[\lambda](i)} \cap \hcl{\Achart[\mu](j)}$, whence $\Achart[\lambda](i) \cbr x \cbr \Achart[\mu](j)$.

Finally, let $\libnd,\mjbnd\in(0;1)$ be arbitrary numbers.
Denote
\begin{align*}
\liComp &= \Achart[\lambda]\bigl((i;\libnd)\bigr), &
\lSegm  &= \Achart[\lambda]\bigl([\lbnd;\libnd]\bigr), &
\mjComp &= \Achart[\mu]\bigl((j;\mjbnd)\bigr), &
\mSegm  &= \Achart[\mu]\bigl([\mbnd;\mjbnd]\bigr).
\end{align*}
Note that
\begin{align*}
    \Wman &= \lComp \cup \{x\} \cup \mComp, &
    \Vman &= \liComp \cup \{x\} \cup \mjComp.
\end{align*}
We need to check the openness of $\Vman$.

Since $\Achart$ is an open embedding, $\liComp$ and $\mjComp$ are open in $\Aone$.
On the other hand, $\lSegm$ and $\mSegm$ are compact subsets of $\Ahaus$, and therefore, by Lemma~\ref{lm:ahaus}\ref{enum:lm:ahaus:compact}, they are closed in $\Aone$.
Now
\begin{itemize}
\item if $\lComp\subset\liComp$ and $\mComp\subset\mjComp$, then $\Vman = \liComp \cup \Wman \cup \mjComp$;
\item if $\lComp\subset\liComp$ and $\mjComp\subset\mComp$, then $\Vman = \liComp \cup (\Wman \setminus \mSegm)$;
\item if $\liComp\subset\lComp$ and $\mComp\subset\mjComp$, then $\Vman = (\Wman \setminus \lSegm) \cup \mjComp$;
\item if $\liComp\subset\lComp$ and $\mjComp\subset\mComp$, then $\Vman = \Wman \setminus (\lSegm \cup \mSegm)$.
\end{itemize}
Thus, in all four cases $\Vman$ is open.

\statement[enum:brpt_bd]
{For each point $x\in\partial\Aone$ there exist an open neighbourhood $\Wman$, a homeomorphism $\Bchart\colon[0;1)\to\Wman$, $\lambda\in\Lambda$, $i\in\partial\Kman_{\lambda}$, $\lbnd\in(0;1)$ such that
\begin{enumerate}[leftmargin=4em, label={\rm(A8-\arabic*)}]
\item\label{enum:brpt_bd:0} $\Bchart(0)=x$;
\item\label{enum:brpt_bd:p} $\lComp := \Achart[\lambda]\bigl((i;\lbnd)\bigr) = \Bchart\bigl((0;1)\bigr) \subset \Acomp[\lambda]$;
\item\label{enum:brpt_bd:ij} $\Achart[\lambda](i) \cbr x$.
\item\label{enum:brpt_bd:another}
for each $\libnd\in(0;1)$ the set $\Vman = \Achart[\lambda]\bigl((i;\libnd)\bigr) \cup \{x\}$ is an open neighbourhood of $x$.
\end{enumerate}}
The proof is completely analogous to~\ref{enum:brpt_int}.

\statement[enum:func_bd]{$\func(\partial\Eman) = \Xvrt$.}
From property~\ref{enum:phil:bd} it follows that $\Achart(\partial\Eman) =
    \mathop{\cup}\limits_{\lambda\in\Lambda}\Achart[\lambda](\partial\Kman_{\lambda})
    \subset \Abr\cup\partial\Aone = \Avrt$,
and therefore
\begin{equation}\label{equ:fdE__piBrBd}
    \func(\partial\Eman)
        = \apr\circ\Achart(\partial\Eman) \subset \apr(\Avrt)
        = \Xvrt.
\end{equation}
Conversely, let $q\in\Xvrt$ and $x\in \apr^{-1}(q) \subset \Abr\cup\partial\Aone$ be an arbitrary point.
Then statements~\ref{enum:brpt_int:ij} and~\ref{enum:brpt_bd:ij} show that there exist $\lambda\in\Lambda$ and $i\in\partial\Kman_{\lambda} \subset \partial\OEman$ such that $x$ is chain inseparable from $\Achart[\lambda](i)$.
Therefore $q = \apr(x) = \apr(\Achart[\lambda](i)) = \func(i) \in \func(\partial\OEman)$.
Thus $\func(\partial\Eman) = \Xvrt$.

\statement[enum:f_quotient_map]
{The map $\func = \apr\circ\Achart\colon\Eman\to\Xone$ is a quotient map.}

By construction, $\func$ is continuous.
Therefore it suffices to show that if $\Bman\subset\Xone$ is a subset such that \emph{$\func^{-1}(\Bman)$ is open in $\Eman$, then $\Bman$ is open in $\Xone$}.
Since $\apr$ is a quotient map, it is enough to check that \emph{$\apr^{-1}(\Bman)$ is open in $\Aone$}.

For each $\lambda\in\Lambda$, set $\Aman_{\lambda} = \func^{-1}(\Bman) \cap \Kman_{\lambda}$.
Since $\func^{-1}(\Bman)$ and the segment $\Kman_{\lambda}$ are open subsets of $\Eman$, we have
\[
    \Aman_{\lambda} =
        \mathop{\sqcup}\limits_{i\in\partial\Kman_{\lambda}} [i; \eps_{i,\lambda})
        \ \bigsqcup \
        \mathop{\sqcup}\limits_{j} (s_{j,\lambda}; t_{j,\lambda})
\]
as a disjoint union of open intervals contained in $\Int{\Kman_{\lambda}} = (0;1)$ and half-open neighbourhoods of the boundary points $\partial\Kman_{\lambda}$.

We need to prove that $\apr^{-1}(\Bman)$ is an open subset of $\Aone$, i.e.\ that each point $x\in\apr^{-1}(\Bman)$ is interior.

1. Suppose that $x \in\apr^{-1}(\Bman)$ also belongs to $\Aedg=\Ahaus\setminus\partial\Aone = \Achart(\Int{\Eman})$.
Then $x = \Achart[\lambda](q)$ for some $\lambda\in\Lambda$ and a point $q \in (s_{j,\lambda}; t_{j,\lambda}) \subset  \Aman_{\lambda}$.
Since $\Achart[\lambda]$ homeomorphically maps $(0;1)$ onto the open subset $\Acomp[\lambda]$ of $\Aone$, the set $\Achart[\lambda]\bigl((s_{j,\lambda}; t_{j,\lambda})\bigr)$ is an open neighbourhood of $x$ in $\Aone$ which is contained in $\apr^{-1}(\Bman)$.
Thus $x$ is interior to $\apr^{-1}(\Bman)$.

2. Now let $x \in\apr^{-1}(\Bman) \cap \Avrt$.

2.1. First assume that $x\in\Int{\Aone}$.
Let $\Wman$, $\Bchart\colon(-1;1)\to\Wman$, $\lambda,\mu\in\Lambda$, $i\in\partial\Kman_{\lambda}$,$j\in\partial\Kman_{\mu}$, $\lbnd, \mbnd\in(0;1)$ be as in~\ref{enum:brpt_int:another}.
Since $\Achart[\lambda](i) \cbr x \cbr \Achart[\mu](j)$, we have
\[
    i,j \in \Achart^{-1}\bigl(\apr^{-1}(\apr(x))\bigr) = \func^{-1}(\apr(x)) \subset \func^{-1}(\Bman).
\]
Hence $[i;\eps_{i,\lambda})\times\{\lambda\}, [j;\eps_{j,\mu})\times\{\mu\} \subset  \func^{-1}(\Bman)$.
But then from~\ref{enum:brpt_int:another} we obtain that the set
\[
    \Vman = \Achart[\lambda]\bigl((i;\eps_{i,\lambda})\bigr) \cup \{x\} \cup \Achart[\mu]\bigl((j;\eps_{j,\mu})\bigr)
\]
is an open neighbourhood of $x$.
It remains to note that $\Vman\subset\apr^{-1}(\Bman)$.

2.2. The case $x\in\partial\Aone$ is proved similarly using~\ref{enum:brpt_bd}.
We leave the details of the proof to the reader.

\statement[enum:f_is_an_atlas]{The composition $\func = \apr\circ\Achart\colon\Eman\to\Xone$ is the desired atlas of an open one-dimensional CW complex with vertex set $\apr(\Abr\cup\partial\Aone)$.}
It suffices to verify the conditions of Lemma~\ref{lm:char:1cw}\ref{enum:lm:char:1cw:quotient}.
By~\ref{enum:f_quotient_map}, the map $\func$ is a quotient map.
Moreover, by~\ref{enum:func_int}, the restriction $\restr{\func}{\Int{\Eman}}\colon\Int{\Eman} \to \Xone$ is injective.
Finally, from~\ref{enum:func_int} and~\ref{enum:func_bd} we obtain that $\func(\partial\Eman) \cap \func(\Int{\Eman}) = \Xvrt \cap \Xedg = \varnothing$.

\statement[enum:minimal_haus_quotient]{$\Xone$ is Hausdorff and for every continuous map $f\colon\Aone\to\Bone$ into a Hausdorff space $\Bone$ there exists a unique continuous map $\hat{f}\colon\Xone\to\Bone$ making the following diagram commute:
\[
\xymatrix{
    \Aone \ar[rr]^{f} \ar[rd]_{\apr[\Aone]} & & \Bone \\
    & \Xone \ar[ur]_{\hat{f}}
}
\]
}

Hausdorffness of CW complexes is well known, and for the one-dimensional case it is shown in Lemma~\ref{lm:1cw_vert_nbh}\ref{enum:lm:1cw_vert_nbh:haus}, while the existence of the map $\hat{f}$ is shown in Lemma~\ref{lm:f_hcl_to_hcl}\ref{enum:lm:f_hcl_to_hcl:hausd}.
\qed
\end{enumerate}

\subsection*{Additional information}
The authors contributed equally to this work. There are no conflicts of interest.

\subsection*{Acknowledgements}
This work was supported by a grant from the Simons Foundation (SFI-PD-Ukraine 00014586, S.I.M. and V.I.Y.).
The authors are sincerely grateful to the anonymous referee for the detailed analysis of the paper, valuable remarks, and indications of the inaccuracies present in the initial version of the paper.


\end{document}